\documentclass[12pt]{article}
\usepackage[utf8]{inputenc}
\usepackage{amsmath}
\usepackage{amssymb}
\usepackage{amsthm}
\usepackage[mathscr]{eucal}
\usepackage{latexsym}
\usepackage{amscd}
\usepackage{graphicx}
\usepackage{epsfig}
\usepackage{float}
\usepackage{hyperref}

\begin{document}

\oddsidemargin0.5truecm
\evensidemargin0.5truecm

\numberwithin{equation}{section}

\theoremstyle{plain}
\swapnumbers\newtheorem{lema}{Lemma}[section]
\newtheorem{prop}[lema]{Proposition}
\newtheorem{coro}[lema]{Corollary}
\newtheorem{teor}[lema]{Theorem}
\newtheorem{fundamental}[lema]{Main Lemma}

\theoremstyle{definition}
\newtheorem{defi}[lema]{Definition}
\newtheorem{defis}[lema]{Definitions}
\newtheorem{ejem}[lema]{Example}
\newtheorem{obse}[lema]{Remark}
\newtheorem{obses}[lema]{Remarks}
\newtheorem{nota}[lema]{Notation}
\newtheorem{sumando}[lema]{Determination of the summand $\kostkasf \la\la \times \kostkasf \mu\mu \times \lr\la\mu\nu$}
\newtheorem{colcon}[lema]{Column conditions}
\newtheorem{rowcon}[lema]{Row conditions}
\newtheorem{redus}[lema]{Redundant inequalities}

\renewcommand{\refname}{\large\bf References}
\renewcommand{\thefootnote}{\fnsymbol{footnote}}

\def\flecha{\longrightarrow}
\def\asocia{\longmapsto}
\def\sii{\Longleftrightarrow}
\def\implica{\Longrightarrow}
\def\implicarev{\Longleftarrow}

\def\vector#1#2{({#1}_1,\dots,{#1}_{#2})}
\def\palabra#1#2#3{{#1}_{#2}\cdots{#1}_{#3}}
\def\lista#1#2{{#1}_1, \dots, {#1}_{#2}}
\def\nume#1{\boldsymbol{[}\,{#1}\,\boldsymbol{]}}
\def\parti{\vdash}
\def\partis#1{{\mathscr{P}({#1})}}
\def\fun#1#2#3{{#1}\,\colon {#2} \flecha {#3}}
\def\suma#1#2#3{\sum_{{#1} = {#2}}^{#3}}
\def\negra#1{\boldsymbol{#1}}

\def\natural{{\mathbb N}}
\def\entero{{\mathbb Z}}
\def\racional{{\mathbb Q}}
\def\real{{\mathbb R}}
\def\complejo{{\mathbb C}}
\def\vacio{\varnothing}

\def\anipol#1#2{{#1}\hspace{.1em}\boldsymbol{[}\hspace{.06em}{#2}\hspace{.06em}\boldsymbol{]}}

\def\ol#1{\overline{#1}}
\def\wt#1{\widetilde{#1}}
\def\wh#1{\widehat{#1}}

\def\cara#1{\chi^{#1}}
\def\permu#1{\phi^{\,#1}}
\def\coefi{{\rm g}(\la,\mu,\nu)}
\def\coefiestable{\ol{\rm g}(\la,\mu,\nu)}
\def\coefili#1#2#3{{\rm g}({#1},{#2},{#3})}
\def\kron{\chi^\lambda\otimes\chi^\mu}
\def\kronli#1#2{\chi^{#1}\otimes\chi^{#2}}

\def\sime#1{S_{#1}}

\def\la{\lambda}
\def\longi#1{\ell({#1})}

\def\menore{\prec}
\def\menori{\preccurlyeq}
\def\mayore{\succ}
\def\mayori{\succcurlyeq}

\def\abc#1#2#3{{#1}_{{#2}{#3}}}
\def\columna#1{{\sf col}({#1})}
\def\contenido#1{{\rm cont}({#1})}
\def\forma#1{{\rm sh}({#1})}
\def\cano#1{C({#1})}
\def\insercol#1#2#3#4{{#1}_{#2} \rightarrow ( \cdots  \rightarrow
({#1}_{#3} \rightarrow {#4}) \cdots )}
\def\inserrow#1#2#3#4{( \cdots ( {#1} \leftarrow {#2}_{#3} ) \leftarrow \cdots )
\leftarrow {#2}_{#4}}
\def\kos#1#2{K_{{#1}{#2}}}
\def\kostkasf#1#2{{\rm ST}_{{#1}{#2}}}

\def\lr#1#2#3{{\rm LR}({#1},{#2};{#3})}
\def\lrest{{\rm LR}^*(\alpha,\beta;\nu)}
\def\lrlibre#1#2#3{{\rm LR}({#1},{#2};{#3})}

\def\transpuesta#1{{#1}\raisebox{.20ex}{$^{\text{\fontsize{7pt}{0pt}\selectfont {\sf T}}}$}}
\def\matriz{{\sf M}(\lambda,\mu)}
\def\numatriz{{\sf m}(\lambda,\mu)}
\def\matrizcon{{\sf M}_{\nu}(\lambda,\mu)}
\def\cardimatrizcon{{\sf m}_\nu(\lambda,\mu)}
\def\partipla{{\sf P}_\nu(\la,\mu)}
\def\matrizd{{\sf M}^*(\lambda,\mu)}
\def\numatrizd{{\sf m}^*(\lambda,\mu)}
\def\matriztres{{\rm M}(\lambda,\mu,\nu)}
\def\mtlibre#1#2#3{{\rm M}({#1}, {#2}, {#3})}
\def\matriztresreal{{\rm L}(\lambda,\mu,\nu)}
\def\mtrlibre#1#2#3{{\rm A}({#1}, {#2}, {#3})}
\def\transporte#1#2#3{{\rm T}({#1}, {#2}, {#3})}
\def\numatres{{\sf m}(\lambda,\mu,\nu)}
\def\matriztresest{{\sf M}^*(\lambda,\mu,\nu)}
\def\matriztresd{{\sf M}^*(\lambda,\mu,\nu)}
\def\numatresd{{\sf m}^*(\lambda,\mu,\nu)}
\def\matriztresdc{{\sf M}^*(\lambda,\mu,\nu^{\,\prime})}
\def\matcol#1{{#1}^{\sf col}}
\def\matrow#1{{#1}^{\sf row}}
\def\multimat#1{\bigl({#1}_{i,j}^{(k)}\bigr)}
\def\multient#1#2#3#4{{#1}_{{#2},{#3}}^{({#4})}}
\def\liri#1#2#3{{\rm lr}({#1},{#2};{#3})}
\def\lirid#1#2#3{{\sf lr}^*({#1},{#2};{#3})}
\def\transptres#1{{\rm A}(\la,\mu,{#1})}

\def\renglon#1{{\sf row}(#1)}
\def\sst{semistandard\ tableau\ }
\def\sucesion#1#2{{#1}_1 < {#1}_2 < \cdots < {#1}_{#2}}
\def\wcol#1{w_{\rm col}({#1})}
\def\insertar#1#2#3{{#1}_{#2} \rightarrow ( \cdots  \rightarrow
({#1}_1 \rightarrow {#3}) \cdots )}

\def\bili#1#2{\langle{#1}, {#2}\rangle}
\def\bilipunto{\bili{\,\cdot\,}{\cdot\,}}

\def\crpol{{\rm CR}(\la,\mu;\nu)}
\def\crpollibre#1{{\rm CR}(\la,\mu;{#1})}
\def\crpollibreestable#1{\ol{\rm CR}(\la,\mu;{#1})}
\def\crcone#1#2#3{{\rm CR}({#1},{#2};{#3})}
\def\posiuno#1#2{P_{#1}(\la,\mu;#2)}
\def\posidos#1#2#3#4{P_{#1,#2}^{(#3)}(\la,\mu;#4)}
\def\policol#1#2#3{C_{#1,#2}(\la,\mu;{#3})}
\def\polirow#1#2#3{R_{#1,#2}(\la,\mu;{#3})}
\def\colineq#1#2{{\rm C}(#1,#2)}
\def\rowineq#1#2{{\rm R}(#1,#2)}
\def\entrada#1#2#3{x_{#1,#2}^{(#3)}}
\def\nivel#1{X^{(#1)}}
\def\notforced{{\rm NF}(p,q,r)}

\begin{centering}
{\Large\bf  On the computation of Kronecker coefficients I: column-row polytopes}\\[.8cm]
{{\large Ernesto Vallejo}\footnotemark[1] \footnotemark[2]}\\
Universidad Nacional Aut\'onoma de M\'exico\\
Centro de Ciencias Matem\'aticas, Morelia, Mexico\\
{\tt vallejo@matmor.unam.mx}\\[.5cm]
{{\large Pedro David S\'anchez Salazar}}\\
Universidad Aut\'onoma de Yucat\'an\\
M\'erida, Mexico\\
{\tt pedro.sanchez@correo.uady.mx}\\
\end{centering}
\footnotetext[1]{Supported by UNAM-DGAPA grant IN108314}
\footnotetext[2]{Corresponding author}

\vskip 1.5pc
\begin{abstract}
We present a way of computing Kronecker coefficients
that uses a new family of rational convex polytopes, called column-row polytopes.
We give several different formulas for the computation.
They are alternating sums of numbers of integer points of either
column-row polytopes or faces of column-row polytopes.
We also compute the maximal dimension of these polytopes and give new
proofs of some known results of more theoretical nature.

\end{abstract}

\tableofcontents

\section{Introduction}\hfill

Kronecker coefficients are one of the most intriguing and challenging families of
numbers in algebraic combinatorics.
They come from the representation theory of the symmetric group $S_n$ and include the celebrated
Littlewood-Richardson coefficients as particular cases.
A Kronecker coefficient $\coefi$ depends on three partitions $\la$, $\mu$ and $\nu$ of the
same integer $n$.
Denote by $\cara\la$ the irreducible character of the irreducible representation $S^\la$ of the
symmetric group $\sime n$ in characteristic zero, by $\otimes$ the product of characters
corresponding to the tensor product of representations,
and by $\bili \cdot \cdot$ the inner (scalar) product of complex characters.
The  \emph{\textbf{Kronecker coefficient}} $\coefi$ is defined as the multiplicity
of $\cara\nu$ in the product of characters $\cara\la \otimes \cara\mu$, that is,
\begin{equation*}
\coefi = \bili {\cara\la \otimes \cara\mu}{\cara\nu}.
\end{equation*}

The aim of these paper is to give a way for computing every Kronecker coefficient
$\coefi$ by counting integer points in a new family of convex polytopes depending on the partitions.
We also give new proofs of some known results using as a tool these polytopes.
This shows that their usefulness is not only computational.
Our approach is based on a method due to Robinson and Taulbee~\cite{rt54}, which
expresses a Kronecker coefficient as an integral linear combination of numbers obtained from
iterated application of the Littlewood-Richardson rule.
Their method goes as follows: Let $\permu\gamma$ denote the character of the permutation
representation $\anipol \complejo {\sime n/ \sime \gamma}$, where
$\sime \gamma$ denotes the Young subgroup of $\sime n$ associated to $\gamma$.
Then
\begin{equation} \label{ecua:robinson-taulbee}
\coefi = \sum_{\gamma \mayori \nu} K_{\gamma\nu}^{(-1)}\, \bili {\cara\la \otimes \cara\mu} {\permu \gamma},
\end{equation}
where the sum runs over the set of partitions $\gamma$ of $n$ that are greater than or equal to $\nu$
in the dominance order and the $K_{\gamma\nu}^{(-1)}$'s are entries in the inverse of the Kostka matrix
for $S_n$.
The number $\bili {\cara\la \otimes \cara\mu} {\permu \gamma}$ can be computed combinatorially
by means of Frobenius reciprocity and the Littlewood-Richardson rule
(for details see Section~2 in~\cite{va99}).
I.~F.~Donin gave a similar combinatorial description of $\bili {\cara\la \otimes \cara\mu} {\permu \gamma}$
in~\cite{don89}.
In Theorem~\ref{teor:cr-poli} we prove
\begin{equation} \label{ecua:liri se cuenta con puntos enteros}
\bili {\cara\la \otimes \cara\mu} {\permu \gamma} = \#  \crpollibre \gamma,
\end{equation}
where $\crpollibre \gamma$ is a rational convex polytope to be defined in
Section~\ref{sec:cr polytopes}  and $\# P$ denotes the number of integer points
in a convex polytope $P$.
We call $\crpollibre \gamma$ a \emph{column-row polytope}.
It is a subset of a $3$-way transportation polytope (see~\cite{delokim} for the definition).
Its integer points are 3-dimensional matrices with nonnegative integer entries
(also called 3-way statistical tables or tensors of order $3$) whose $1$-marginals
(or plane sums) are $\la$, $\mu$ and $\gamma$.
The inequalities describing these polytopes are new, but dual (in the sense of the
Robinson-Schensted-Knuth correspondence) to the inequalities that appeared in~\cite{va03}.
One should \emph{stress} that these inequalities are not obtained from a combination of inequalities
defining Littlewood-Richardson coefficients; they are essentially different.
Combining~\eqref{ecua:robinson-taulbee} and~\eqref{ecua:liri se cuenta con puntos enteros}
we get our first main result, Theorem~\ref{teor:kron-alternante}, which states that
\begin{equation} \label{ecua:rotapo1}
\coefi = \sum_{\gamma \mayori \nu} K_{\gamma\nu}^{(-1)}\, \# \crpollibre \gamma.
\end{equation}
This sum can also be computed by means of the Jacobi-Trudi determinant that expresses
the irreducible character $\cara\nu$ as a linear combination of the permutation
characters $\phi^\gamma$.
The right-hand side of~\eqref{ecua:rotapo1} can be then computed via Barvinok's algorithm~\cite{ba08, bp99}.
If $\la$, $\mu$ and $\gamma$ have lengths $p$, $q$ and $r$, respectively,
we show in Theorem~\ref{teor:dim-crpolytope} that the dimension of $\crpollibre \gamma$ is at most
\begin{equation*}
\textstyle
pqr - \binom{p}{2} - \binom {q}{2} - (p+q+r) + 2.
\end{equation*}
This is our second main result.

Current applications of Kronecker coefficients in quantum information and in geometric
complexity theory require more efficient algorithms.
There are other algorithms that exploit the power of present day computers.
M.~Christandl, B.~Doran and M.~Walter~\cite[Prop.~VI.1]{cdw12} use the theory of Lie groups and algebras
to express Kronecker coefficients as an integral linear combination of numbers of integral points
in $3$-way transportation polytopes.
V.~Baldoni and M.~Vergne~\cite{bavew18} use, besides Lie groups and Lie algebra theory,
methods from symplectic geometry and residue calculus to compute individual Kronecker coefficients
as well as symbolic formulas that are valid on an entire polyhedral chamber.
M.~Mishna, M.~Rosas, S.~Sundaram and S.~Trantafir use the theory of symmetric
functions to express Kronecker coefficients as a integral linear combination of vector partition
functions (see~\cite{mrs} and \cite[Thm.~1.1]{mitr}).

Our method is of a very different nature than the previous ones.
It is pretty basic in the sense that it only requires identity~\eqref{ecua:kron-rt},
some results on the combinatorics of Young tableaux contained in~\cite{ful} and a bijection
between $3$-dimensional matrices and certain triples of tableaux given in~\cite{av12}.
The column-row polytopes appearing in equation~\eqref{ecua:rotapo1} are defined
by explicit equalities (see Propositions~\ref{prop:column-conditions} and~\ref{prop:row-conditions}).
The combinatorial proof of the stability of Kronecker coefficients given in~\cite{va99}
and based in equation~\eqref{ecua:kron-rt} can be translated to a proof of the stability
using equation~\eqref{ecua:rotapo1}.
With an adequate parametrization of each column-row polytope one can give better bounds
for the stability of a specific Kronecker coefficient and, more importantly, one obtains
a formula for the reduced Kronecker coefficients of the form
\begin{equation*}
\coefiestable = \sum_{\gamma \mayori \nu} K_{\gamma\nu}^{(-1)}\, \# \crpollibreestable \gamma,
\end{equation*}
where each $\crpollibreestable \gamma$ is defined by a subset of the inequalities
defining $\crpollibre \gamma$.
This is carried out in~\cite{biva}.

In~\cite[Example~3.20]{bavew18} V. Baldoni and M. Vergne have shown that, for partitions
$\la$, $\mu$ and $\nu$ with $p$, $q$ and $r$ parts, respectively, the coefficient
$\coefili {t \la}{t \mu}{t \nu}$ is a quasi-polynomial function of $t$ of degree at most
\begin{equation} \label{ecua:dimension balver}
\textstyle
pqr - \binom p2 - \binom q2 - \binom r2 - (p+q+r) + 2.
\end{equation}
Thus, the dimension of the column-row polytope given in Theorem~\ref{teor:dim-crpolytope}
exceeds by $\binom r2$ the degree of the corresponding quasi-polynomial function.
One can in fact adjust formula~\eqref{ecua:rotapo1} and produce another that involves only some
faces of the column-row polytopes.
This is the content of Theorem~\ref{teor:kron-eficiente}, our third main result.
Since Barvinok's algorithm runs in $L^{O(d \log d)}$ time for a polytope of  dimension $d$
whose input has size $L$ (see~\cite[Thm.~2.5]{papa17}), this adjustment reduces both the dimension
and the input size of the polytopes involved in the computation of a Kronecker coefficient,
hence the time needed to compute it.
With the existing algorithms one can compute Kronecker coefficients on a small computer
in a span of time that goes from a few seconds to a couple of hours only
for $r=2$ and $3$ and small values of $p$ and $q$ (see~\cite{bavew18, biva, mrs, mitr}).
In the case $r=2$ the formula in Theorem~\ref{teor:kron-eficiente} permits one to compute
Kronecker coefficients by counting integer points in polytopes whose dimension does not
exceed~\eqref{ecua:dimension balver}.
We expect to obtain an analogous formula for the case $r=3$.
It should be underlined that there is something mysterious in the nature of column-row polytopes
in that one can compute the same Kronecker coefficient using different faces of the same
polytopes depending on the value of $\ell$ in Theorem~\ref{teor:kron-eficiente}.

The paper is organized in the following way:
In Section~\ref{sec:coef-de-kron} we review the notions related to Young tableaux
and the representation theory of the symmetric group.
In Section~\ref{sec:insertion of words} we present the Main Lemma concerning
insertions of words in Young tableaux which will yield the inequalities defining
the column-row polytopes.
Its proof is deferred to Section~\ref{sec:proofs}.
The other ingredient in the definition of column-row polytopes is a one-to-one
correspondence between 3-dimensional matrices and certain triples of tableaux
appearing in~\cite{av12}.
This is explained in Section~\ref{sec:correspondencia}.
We define our main tool for computing Kronecker coefficients, namely column-row polytopes,
in Section~\ref{sec:cr polytopes}.
Here, we express any Kronecker coefficient as an integral linear combination of numbers
of integer points in these polytopes (Theorem~\ref{teor:kron-alternante}).
We also present some properties of column-row polytopes as well as new proofs of some
known results.
These polytopes are intersections of hyperplanes depending on $\la$, $\mu$ and $\nu$ with a convex cone
depending only on the lengths of $\la$, $\mu$ and $\nu$ (called the \emph{column-row cone}).
We determine the dimensions of column-row cones and column-row polytopes in Sections~\ref{sec:cr-cones}
and~\ref{sec:dim column-row polytope}, respectively.
Finally, in Section~\ref{sec:column and row ineqs} we give an alternative
description of some of the inequalities defining the column-row cone with the aim of
giving another more efficient way, in Section~\ref{sec:red of dim}, of computing Kronecker
coefficients by counting integer points in some faces of the column-row polytopes.
This is done in Theorem~\ref{teor:kron-eficiente}.

\section{Kronecker coefficients} \label{sec:coef-de-kron} \hfill

In this paper we use the following notation:
$\natural$ is the set of positive integers.
For any $m$, $n\in \natural \cup \{0\}$,  $\nume n = \{1, \dots, n\}$ and $\nume {m,n}$
is the interval of integers $k$ such that $m \le k \le n$.
A \emph{\textbf{partition}} $\la = \vector \la p$ is a weakly decreasing sequence
$\la_1 \ge \la_2 \ge \cdots \ge \la_p$ of nonnegative integers.
Its \emph{\textbf{size}} is $|\la | = \la_1 + \cdots + \la_p$ and
its \emph{\textbf{length}} $\longi{\la}$ is the number of positive parts.
If $|\la|=n$, we also write $\la\vdash n$ and say that $\la$ is a partition of $n$.
A \emph{\textbf{composition}} of $n$ is a sequence of positive integers $\gamma = \vector \gamma k$
such that $\sum_{i=1}^k \gamma_i = n$.
Let $\mu$ be a partition whose diagram is contained in the diagram of $\la$;
we write $\la/\mu$ for the \emph{\textbf{skew partition}} formed by the boxes in
$\la$ that are not in $\mu$.
A \emph{\textbf{semistandard tableau}} $T$ is a filling of a skew diagram of $\la/\mu$ with
numbers from $\natural$ that is weakly increasing in rows from left to right and
strictly increasing in columns from top to bottom (in English notation).
Its \emph{\textbf{shape}}, denoted $\forma T$, is $\la/\mu$ and its
\emph{\textbf{content}}, denoted $\contenido T$, is the composition $\vector \gamma n$,
where $\gamma_i$ is the number of $i$'s in $T$.
The \emph{\textbf{word}} (or \emph{\textbf{row word}}) of a skew tableau $T$, $w_{\rm row}(T)$,
is defined by reading the entries of $T$ from left to right and from bottom to top.
A word $\palabra w1n$ is called a \emph{\textbf{reverse lattice word}} if when read
backwards from the end to any letter, the sequence $\palabra wnk$ contains at least
as many $i$'s as it does $i+1$'s, for any $i \in \natural$.
A \emph{\textbf{Littlewood-Richardson skew tableau}} is a semistandard tableau $T$
whose word $w_{\rm row}(T)$ is a reverse lattice word.
For more details on the combinatorics of Young tableaux and the representation theory of the symmetric group
we refer the reader to~\cite{ful, m95, sa01, st99}.

Kronecker coefficients generalize Littlewood-Richardson coefficients for which there
are various combinatorial interpretations.
The first one of them goes back to 1934 (see~\cite{liri34}).
Ideally, one would also like a combinatorial interpretation for Kronecker coefficients
(see Problem~10 in~\cite{stan00}).
Recently, doubts have been cast on the existence of such a interpretation (\cite{ikenpak, pak-intp-2022}).
If indeed, a combinatorial interpretation for Kronecker coefficients does not exist
the need of other methods of computing them, like the one we are presenting, would be even more
necessary.
The quest for efficient ways of computing Kronecker coefficients has been an open
problem since the appearance of the first paper on the subject more than 85 years ago~\cite{mu38}.
It is well-known~\cite{don89, rt54} that Kronecker coefficients
can be obtained as an alternating sum of numbers computed combinatorially.
A way of doing this is the following:

Let $\tau = \vector \tau r$ be a composition of $n$.
A sequence $T = \vector Tr$ of tableaux is called a \emph{\textbf{Littlewood-Richardson multitableau}}
of \emph{\textbf{shape}} $\la$ and \emph{\textbf{type}} $\tau$ if there exists a
sequence of partitions $\la(0)$, $\la(1),  \dots ,\la(r)$ such that
\begin{equation*}
\vacio = \la(0) \subseteq \la(1) \subseteq \cdots
\subseteq \la(r) = \la
\end{equation*}
and $T_i$ is a Littlewood-Richardson tableau of shape $\la(i)/\la(i -1)$ and size $\tau_i$
(the number of squares of $T_i$) for all $i\in \nume r$.
If each $T_i$ has content $\gamma(i)$, then we say that $T$ has \emph{\textbf{content}}
$(\gamma(1),\dots, \gamma(r))$.
Since $T_i$ is a Littlewood-Richardson tableau, $\gamma(i)$
is a partition of $\tau_i$.
We denote by $\lr \la\mu\tau$ the set of all pairs $(T,S)$ of Littlewood-Richardson
multitableaux such that $T$ has shape $\la$, $S$ has shape $\mu$ and both have the same
content $\tau$.
The number of such pairs is denoted by
\begin{equation} \label{ecua:liri tableau}
\liri \la\mu\tau = \# \lr \la\mu\tau.
\end{equation}

Let $\permu\tau$ denote the permutation character associated to $\tau$.
It is well known (see for example Section~2 in~\cite{va99} or Lemma 4.1 in~\cite{va09}) that
\begin{equation} \label{ecua:liri caracter}
\liri \la\mu\tau = \langle \cara\la \otimes \cara\mu, \permu\tau \rangle .
\end{equation}
Note that for any reordering $\sigma$ of $\tau$, since $\permu\sigma = \permu\tau$, one has
\begin{equation} \label{ecua:liri number}
\liri \la\mu\sigma = \liri \la\mu\tau.
\end{equation}

Those readers more familiar with symmetric functions can translate, via the Frobenius
characteristic map (\cite[Section~I.7]{m95} or~\cite[Section~7.18]{st99}),
the identities involving characters of the symmetric group into identities
involving symmetric functions.
Thus, if $s_\la$ denotes a Schur function, $h_\la$ a complete symmetric function, $*$
the internal product of symmetric functions and $\bilipunto$ the scalar product of
symmetric functions, one has
\begin{equation*}
\coefi = \bili {s_\la * s_\mu}{s_\nu}
\end{equation*}
and
\begin{equation*}
\liri \la\mu\tau = \bili {s_\la * s_\mu}{h_\tau}.
\end{equation*}

The \emph{\textbf{Kostka number}} $\kos\la\tau$ is the number of semistandard tableaux
of shape $\la$ and content $\tau$.
Let us arrange the partitions of $n$ in reverse lexicographic order
and form the matrix $K_n =(\kos\la\mu )$.
It is invertible over the integers.
The entries of the inverse matrix $K_n^{-1}= \Bigl( \kos\la\mu^{(-1)} \Bigr)$ can be computed
either from Jacobi-Trudi determinants or combinatorially~\cite{er90, m95}.
Thus every irreducible character can be expressed as in integral linear combination
of permutation characters.
Then, from the bilinearity of the inner product of characters, we get
a combinatorial description of Kronecker coefficients:
\begin{equation} \label{ecua:kron-rt}
\coefi = \sum_{\gamma \mayori \nu} \kos\la\gamma^{(-1)}\, \liri \la\mu\gamma.
\end{equation}
Here, $\mayori$ denotes the dominance order of partitions.
The right side of the equation involves, in general, positive as well as negative terms.
This method for computing Kronecker coefficients goes back to Robinson and Taulbee~\cite{rt54}.

The numbers $\liri \la\mu\gamma$ have already been used to obtain results about Kronecker coefficients
(see, for example~\cite{av12, va99, va00, va09, va14}).
We will show in Section~\ref{sec:cr polytopes} that $\liri \la\mu\gamma$ equals the number of
integer points in a column-row polytope.
Hence, one can use Barvinok's algorithm (\cite{ba08, bp99}) to compute Kronecker coefficients.

\section{Insertion of words} \label{sec:insertion of words} \hfill

Column-row polytopes are subsets of $3$-way transportation polytopes
defined by a set of equalities and inequalities obtained in the Main Lemma~\ref{lema:cano-cano}.
The proofs of the next two lemmas are given in~Section~\ref{sec:proofs}.

For any partition $\la$, the \emph{\textbf{canonical tableau}} $\cano \la$ of shape $\la$
is the unique semistandard tableau of shape and content $\la$.
Given a semistandard tableau $T$, let $\wcol T$ denote the \emph{\textbf{column  word}}
of $T$, that is, the word obtained from $T$ by reading its entries from bottom to top
(in English notation), in successive columns, starting in the left column and moving to
the right (see~\cite[p.~27]{ful}).
For example $\wcol{\cano{3,2,1}} = 321211$.

Let $w=w_1 \cdots w_l$ be a word in the alphabet $\natural$.
We denote by $P(w)$ the tableau obtained by successively column-inserting the
letters of $w$ in the empty tableau $\vacio$:
\begin{equation*}
P(w)= \insercol w1l \vacio.
\end{equation*}
Note that $P(w)$ can also be obtained by row insertion:
\begin{equation*}
P(w) = \inserrow \vacio w1l.
\end{equation*}
There is a neat description of $w$ when $P(w)$ is a canonical tableau.

\begin{lema} \label{lema:cano-palabra}
Let $w$ be a word of content $\la$.
Then $P(w)= \cano\la$ if and only if $w$ is a reverse lattice word.
\end{lema}

The following lemma and its corollary will be central in our method for computing
Kronecker coefficients through the counting of integer points in polytopes.
If $(P,Q)$ is the pair of semistandard tableaux associated to a matrix $B$ under the
RSK-correspondence, we characterize with linear inequalities when $P$ or $Q$
are canonical tableaux.
The reverse lattice word condition of Lemma~\ref{lema:cano-palabra} can be translated
into a set of inequalities in a similar way as it was done in~\cite{pv05, va03}.

\begin{fundamental} \label{lema:cano-cano}
Let $\la$ and $\mu$ be partitions of the same size and of lengths $p$ and $q$, respectively.
Let $m = \min \{ p, q \}$ and $B = (b_{i,j})$ be a matrix with nonnegative integer entries,
row-sum vector $\la$ and column-sum vector $\mu$.
If $(P,Q)$ is the pair of semistandard tableaux associated to $B$ under the RSK-correspondence, then

(1) $P = \cano \mu$ if and only if

\hspace{15 pt} (i) $b_{i,j} = 0$ for $i+j > p+1$; and

\hspace{15 pt} (ii) $\sum_{k=i}^{p+1-j} b_{k,j} \ge \sum_{k=i-1}^{p-j} b_{k,j+1}$ for $j \in \nume {m-1}$
and $i \in \{ 2, \dots, p+1-j \}$.

(2) $Q = \cano \la$ if and only if

\hspace{15 pt} (i) $b_{i,j} = 0$ for $i+j > q+1$; and

\hspace{15 pt} (ii) $\sum_{l=j}^{q+1-i} b_{i,l} \ge \sum_{l=j-1}^{q-i} b_{i+1,l}$ for
$i \in \nume {m-1}$ and $j \in \{ 2 ,\dots, q+1-i \}$.
\end{fundamental}

\begin{coro} \label{coro:cano-cano}
Let $\la$ and $\mu$ be partitions of the same size and of lengths $p$ and $q$,
respectively.
Let $m = \min \{ p, q \}$ and $B = (b_{i,j})$ be a matrix with nonnegative integer entries,
row-sum vector $\la$ and column-sum vector $\mu$.
Then $( \cano\mu, \cano\la)$ is the pair of semistandard tableaux associated to $B$ under the
RSK correspondence if and only if

(i) $b_{i,j} = 0$ for $i+j > m+1$; and

(ii) for each $k \in \nume m$ the entries $b_{i,j}$ in the diagonal $i+j = k+1$
are all equal.

\noindent
In other words the first $m$ diagonals of $B$ are constant (and different from zero) and the remaining
entries are all zero.
\end{coro}

\section{Matrices and tableaux}\label{sec:correspondencia}\hfill

The second ingredient we need to define column-row polytopes is Theorem~\ref{teor:rsk-3},
which is explained in this section.

Let $\la = \vector \la p$ and $\mu = \vector \mu q$ be partitions of $n$ and let $\tau = \vector \tau r$
be a composition of $n$.
Denote by $\mtrlibre \la \mu \tau$ the affine space of all $3$-dimensional matrices
with real entries having 1-marginals (plane sums)
$\la$, $\mu$ and $\tau$, namely, matrices $A=(a_{ijk})$ such that their entries satisfy
\begin{equation*}
\sum_{y,z} a_{iyz}=\la_i, \quad \sum_{x,z} a_{xjz}=\mu_j \quad
\textrm{and} \quad  \sum_{x,y}    a_{xyk}=\tau_k ,
\end{equation*}
for all $i\in \nume p$, $j \in \nume q$  and $k \in \nume r$.
Let $\transporte \la\mu\tau$ be the set of matrices in $\mtrlibre \la \mu \tau$ with
nonnegative entries.
It is a convex polytope called $3$-way transportation polytope (see~\cite{delokim}).
Let $\mtlibre \la \mu \tau$ denote the set of matrices in $\transporte \la\mu\tau$ with
integer entries.
Its elements are also called $3$-way tables or tensors of order $3$.
Note that we choose to work with $\la$ and $\mu$ as partitions and with $\tau$ as a composition
because we want to describe the elements in the set $\lr \la\mu\tau$ from Section~\ref{sec:coef-de-kron}
by means of matrices in $\mtlibre \la \mu \tau$.
The invariance under permutation of coordinates of $\tau$ shown in~\eqref{ecua:liri number}
yields a great flexibility when applying formula~\eqref{ecua:rotapo1}.
For any partition $\alpha$ of $n$ let $\kostkasf \alpha\la$ denote
the set of all semistandard tableaux of shape $\alpha$ and content $\la$.
The following theorem was proved in~\cite{av12}:

\begin{teor} \label{teor:rsk-3}
There is a one-to-one correspondence between the set $\mtlibre \la \mu \tau$
and the disjoint union of triples
$\bigsqcup_{\,\alpha,\,\beta\vdash n} \kostkasf \alpha\la\times\kostkasf\beta\mu
\times\lr \alpha\beta\tau$.
\end{teor}

Note that since $\la$ and $\mu$ are partitions, $\alpha$ and $\beta$ run over all partitions of
$n$ such that $\alpha \mayori \la$ and $\beta \mayori \mu$ in the dominance order.
We need the explicit description of the bijection:
let $A=(a_{ijk})$ be an element in $\mtlibre \la \mu \tau$.
First, we split $A$ into its \emph{\textbf{level matrices}}
$A^{(k)} = \bigl( a_{ij}^{(k)}\bigr)$, $k\in \nume r$, where $a_{ij}^{(k)}= a_{ijk}$.
For each $k\in \nume r$, let $(P_k,Q_k)$ be the pair of semistandard Young tableaux
associated to $A^{(k)}$ under the RSK correspondence~\cite{ful, kn70, st99}.

To the sequence of tableaux $\vector Pr$ we associate a pair $(P,S)$ in the following way:
let $\gamma(k) =\forma{P_k}$, $k\in \nume r$.
Let $P^{(1)} =P_1$ and $S_1$ be the canonical tableaux $\cano{\gamma(1)}$ of shape $\gamma(1)$.
We define $P^{(k)}$ and $S_k$ inductively:
let $\wcol {P_k}=v_m\cdots v_1$ and $\wcol {\cano{\gamma(k)}}=u_m\cdots u_1$.
The tableau $P^{(k)}$ is obtained by column inserting  $v_1, \dots, v_m$ in
$P^{(k-1)}$, that is,
\begin{equation*}
P^{(k)}=\insertar vm{P^{(k-1)}},
\end{equation*}
and $S_{k}$ is the tableau obtained by placing  $u_1, \dots, u_m$
successively in the new boxes.
Let $P = P^{(r)}$ and $S =\vector Sr$.
Thus $P$ is a semistandard tableau, $S$ is a Littlewood-Richardson
multitableau of content $(\gamma(1),\dots,\gamma(r))$,
$\forma P = \forma S$, and $\contenido P=\sum_{k=1}^r
\contenido{P_k}$.
Note that $P=P_r \boldsymbol{\cdot} \cdots \boldsymbol{\cdot} P_1$, the product of
tableaux as defined in Fulton's book~\cite{ful}.

Similarly we construct $(Q,T)$ out of $\vector Qr$.
Then the correspondence of Theorem~\ref{teor:rsk-3} is given by
\begin{equation} \label{ecua:corresp}
A=(a_{ijk}) \longleftrightarrow (Q,P,(T,S))
\end{equation}

The symmetry of the RSK-correspondence induces a symmetry of bijection~\eqref{ecua:corresp}:
given $A = (a_{ijk})$, denote $\transpuesta A = (b_{ijk})$, where $b_{ijk} = a_{jik}$.
Then, if $A$ corresponds to $(Q,P,(T,S))$, $\transpuesta A$ corresponds to $(P,Q, (S,T))$.

\section{Column-row polytopes} \label{sec:cr polytopes} \hfill

Let $\la$, $\mu$ and $\tau$ be as in Section~\ref{sec:correspondencia}.
Here we apply Theorem~\ref{teor:rsk-3} and the Main Lemma~\ref{lema:cano-cano}
to derive a set of linear equalities and inequalities,
depending on $\la$, $\mu$ and $\tau$, that define a convex polytope whose number of integer points is
$\liri \la\mu\tau$.
More precisely, in the next paragraph, we will describe---by means of linear equalities
and inequalities---all matrices $A\in \mtlibre \la \mu \tau$ that correspond
under bijection~(\ref{ecua:corresp}) to elements in the component
$\kostkasf \la\la \times \kostkasf \mu\mu \times \lr\la\mu\tau$ given by $\alpha=\la$
and $\beta = \mu$ in the disjoint union
\begin{equation*}
\bigsqcup_{\alpha \mayori \la,\,\beta\mayori \mu} \kostkasf \alpha\lambda\times\kostkasf\beta\mu
\times\lr \alpha\beta\tau.
\end{equation*}
Since $\kostkasf \la\la = \{ \cano\la \}$ and $\kostkasf \mu\mu = \{ \cano\mu \} $,
this component has cardinality $\liri \la\mu\tau$.

Let $A=(a_{i,j,k})$ be an element in $\mtlibre \la \mu \tau$ and denote by
$(Q,P, (T,S))$ its image under bijection~(\ref{ecua:corresp}).
For each $k\in \nume r$ let $A^{(k)}$ be the $k$-th level
matrix of $A$, and let $(P_k, Q_k)$ be the pair of semistandard tableaux associated to
$A^{(k)}$ under the RSK correspondence.
Let also
\begin{equation*}
\left(
\begin{matrix}
u_1^{(k)} & \dots & u_{\nu_k}^{(k)} \\
v_1^{(k)} & \dots & v_{\nu_k}^{(k)}
\end{matrix}
\right)
\end{equation*}
be the two-row lexicographic array associated to $A^{(k)}$ as in Section~4.1 in~\cite{ful}.
We consider the words $v^{(k)} = v_1^{(k)} \dots v_{\nu_k}^{(k)}$ and
$v= v^{(r)} \cdots v^{(1)}$.
Then, by the RSK correspondence, $P_k = P(v^{(k)})$.
It follows from definition of the bijection in Theorem~\ref{teor:rsk-3} that
$P = P_r \boldsymbol{\cdot} \cdots \boldsymbol{\cdot} P_1 = P(v)$.

Let $\matcol A =(b_{i,j})$ be the matrix of size $pr \times q$ with blocks
$A^{(r)}, \dots , A^{(1)}$, namely,
\begin{equation} \label{ecua:definicion de A-col}
\matcol A =
\begin{bmatrix}
A^{(r)} \\
\vdots \\
A^{(1)}
\end{bmatrix},
\end{equation}
and let
\begin{equation*}
\left(
\begin{matrix}
s_1 & \dots & s_n \\
t_1 & \dots & t_n
\end{matrix}
\right)
\end{equation*}
be its associated two-row lexicographic array.
Thus $v= t_1 \cdots t_n$.
This implies that if $(R, U)$ is the pair of tableaux associated to $\matcol A$
under the RSK correspondence, $P = R$.
Thus, from part (1) in the Main Lemma~\ref{lema:cano-cano}, we get

\begin{prop} \label{prop:column-conditions}
Let $A \in \mtlibre \la\mu\tau$ and let $(Q, P, (T,S))$ be its image under bijection~(\ref{ecua:corresp});
let $m = \min \{ pr, q \}$.
Then $P=\cano\mu$ if and only if the entries $b_{i,j}$ of $\matcol A$
satisfy the following conditions:

(C1) $b_{i,j}=0$ for all $i$, $j$ such that $i+j > pr+1$; and

(C2) $\sum_{k = i}^{pr + 1 - j} b_{k,j} \ge \sum_{k = i - 1}^{pr - j} b_{k,j+1}$
for all $j\in\nume {m-1}$ and all $i \in \{2, \dots, pr + 1 - j \}$.
\end{prop}

Now, since $Q = Q_r \boldsymbol{\cdot} \cdots \boldsymbol{\cdot} Q_1$,
the symmetry property of the RSK correspondence implies that $Q$
is the insertion tableaux of the matrix
\begin{equation*}
\widetilde A =
\begin{bmatrix}
\transpuesta { \bigl( A^{(r)} \bigr) }\, \\
\vdots \\
\transpuesta { \bigl( A^{(1)} \bigr) }\,
\end{bmatrix}.
\end{equation*}
Therefore, by the same reasoning we applied to $P$, we have that $Q=\cano\la$ if and only if
the entries of $\widetilde A$ satisfy conditions (C1) and (C2) in Proposition~\ref{prop:column-conditions}.
A more convenient way to present these inequalities is to use the transpose of $\widetilde A$,
that is, the matrix $\matrow A$ of size $p \times qr$ defined by
\begin{equation} \label{ecua:definicion de A-row}
\matrow A =
\begin{bmatrix}
A^{(r)} & \cdots & A^{(1)}
\end{bmatrix}.
\end{equation}
Thus we get

\begin{prop} \label{prop:row-conditions}
Let $A \in \mtlibre \la\mu\tau$ and let $(Q, P, (T,S))$ be its image under bijection~(\ref{ecua:corresp});
let $m = \min \{ p, qr \}$.
Then $Q=\cano\la$ if and only if the entries $c_{i,j}$ of $\matrow A$ satisfy
the following conditions:

(R1) $c_{i,j}=0$ for all $i$, $j$ such that $i+j > qr + 1$; and

(R2) $\sum_{l = j}^{qr + 1 -i} c_{i,l} \ge \sum_{l = j - 1}^{qr-i} c_{i+1,l}$
for all $i \in \nume {m-1}$ and all $j \in \{2, \dots, qr + 1 - i\}$.
\end{prop}

These results motivate next definition.

\begin{defi} \label{defi:cr-pol}
Let $A$ be a 3-dimensional matrix of size $p \times q \times r$ with nonnegative
real entries.
We say that $A$ satisfies the \emph{\textbf{vanishing conditions}} if the entries of $\matcol A$
satisfy~(C1) and the entries of $\matrow A$ satisfy~(R1);
we say that $A$ satisfies the \emph{\textbf{column inequalities}} if the entries of $\matcol A$
satisfy~(C2); and we say that $A$ satisfies the \emph{\textbf{row inequalities}} if
the entries of $\matrow A$ satisfy~(R2).
We denote by $\crpollibre \tau$ the set of elements in the $3$-way transportation
polytope $\transporte \la\mu\tau$ that satisfy the vanishing conditions and the column and row inequalities.
This is a rational convex polytope which we call \emph{\textbf{column-row polytope}}.
\end{defi}

For any subset $S$ of a convex polytope $P$, we denote by $\# S$ the number of integer points in $S$.
Then Propositions~\ref{prop:column-conditions} and~\ref{prop:row-conditions} imply the following

\begin{teor} \label{teor:cr-poli}
Let $\la$, $\mu$ be partitions of $n$, let $\tau$ be a composition of $n$.
Then the correspondence~(\ref{ecua:corresp}) restricts to a one-to-one correspondence between
the set of integer points in $\crpollibre \tau$ and the set $\lrlibre \la\mu\tau$.
Therefore
\begin{equation*}
\liri \la\mu\tau = \# \crpollibre \tau.
\end{equation*}
\end{teor}

One of our main results is a consequence of Theorem~\ref{teor:cr-poli} and equation~\eqref{ecua:kron-rt}:

\begin{teor} \label{teor:kron-alternante}
Let $\la$, $\mu$, $\nu$ be three partitions of some integer $n$.
Then
\begin{equation} \label{ecua:kron-alternante}
\coefi = \sum_{\gamma \mayori \nu} K_{\gamma\nu}^{(-1)}\, \# \crpollibre \gamma.
\end{equation}
\end{teor}

\begin{obse} \label{obse:teor-kron-alternante}
Because of identity~\eqref{ecua:liri number} and Theorem~\ref{teor:cr-poli},
one can replace any $\gamma$ in the summation above with any composition obtained by permuting its coordinates.
\end{obse}

We get the following geometric upper bounds for Kronecker coefficients.

\begin{coro} \label{coro:kronecker bounds}
Let $\la$, $\mu$, $\nu$ be partitions of $n$.
Then,
\begin{equation*}
\coefi \le \# \crpol \le \# \transporte \la\mu\nu.
\end{equation*}
\end{coro}
\begin{proof}
From Young's rule, \eqref{ecua:liri caracter} and~Theorem~\ref{teor:cr-poli} one has
\begin{equation} \label{ecua:kron menor o igual que liri}
\coefi \le \langle \cara\la \otimes \cara\mu, \permu\nu \rangle = \liri \la\mu\nu = \# \crpollibre \nu.
\end{equation}
The second inequality follows directly from the definition of $\crpollibre \nu$.
\end{proof}

\begin{obse}
The inequality $\coefi \le \# \transporte \la\mu\nu$ is apparent from identity~(6) in~\cite{av12}.
It was made explicit in Lemma~1.2 in~\cite{papa20}.
The inequality~\eqref{ecua:kron menor o igual que liri} is also apparent from identity~(10)
in~\cite{av12}.
The novelty in Corollary~\ref{coro:kronecker bounds} is that now this upper bound is given in
terms of the number of integer points in a polytope defined explicitly.
In~\cite{av12} the authors started to use $3$-way contingency tables to gain combinatorial
information about Kronecker coefficients.
The correspondence proved there is one of the main ingredients in the construction of the
column-row polytopes.
\end{obse}

\begin{coro}
Let $\la$, $\mu$, $\nu$ and $\gamma$ be partitions of $n$.
If $\gamma \mayori \nu$ in the dominance order, then
\begin{equation*}
\# \crpollibre \gamma \le \# \crpol.
\end{equation*}
\end{coro}
\begin{proof}
This is a consequence of Theorem~\ref{teor:cr-poli} and Theorem~4.10 in~\cite{va14}.
\end{proof}

Extremal components of $\cara\la \otimes \cara\mu$ with respect to the dominance order
can be computed by counting integer points in polytopes.

\begin{coro}
Let $\nu$ be a partition of $n$.

(1) If $\chi^\nu$ is a maximal component (in the dominance order) of $\chi^\la \otimes \chi^\mu$,
then
\begin{equation*}
\coefi =  \# \crpol.
\end{equation*}

(2) If  $\chi^\nu$ is a minimal component (in the dominance order) of $\chi^\la \otimes \chi^\mu$,
then
\begin{equation*}
\coefi = \coefili {\la^{\,\prime}} \mu {\nu^{\,\prime}} =
\# {\rm CR}(\la^{\,\prime},\mu;\nu^{\,\prime}).
\end{equation*}
\end{coro}
\begin{proof}
This is a direct consequence of Lemma~3.2 in~\cite{av12}, Theorem~\ref{teor:cr-poli} and the fact
that taking conjugates is an antiautomorphism in the lattice of partitions under the dominance order.
\end{proof}

We can also give a new proof of the following result of A.~Regev (Thm.~12 in ~\cite{reg80}).

\begin{coro} \label{coro:regev}
Let $\la$, $\mu$, $\nu$ be partitions of $n$ of lengths $p$, $q$, $r$, respectively.
If $\coefi >0$, then $q \le pr$.
\end{coro}
\begin{proof}
Suppose $\coefi >0$ and $q >pr$.
We get a contradiction.
By Corollary~\ref{coro:kronecker bounds} the polytope $\crpol$ is not empty.
Let $X \in \crpol$.
By (C1) in Proposition~\ref{prop:column-conditions} the $q$-th column of $\matcol X$,
as defined in~\eqref{ecua:definicion de A-col}, is equal to zero.
But, since the $\mu_q>0$, this column should not be zero: a contradiction.
\end{proof}

\begin{obse}
Let $p$, $q$, $r$ be the lengths of $\la$, $\mu$, $\nu$, respectively.
In order to have as few as possible summands in equation~\eqref{ecua:kron-alternante},
we make use of the symmetry of Kronecker coefficients and assume without loss
of generality that $r \le p$, $q$.
Also by symmetry we assume $p \le q$.
Moreover, if $\coefi > 0$ one has, by Corollary~\ref{coro:regev}, $q \le pr$ .
Besides, when $r=1$ the computation of the Kronecker coefficient is trivial.
So, we assume $2 \le r$.
To summarize, when dealing with Kronecker coefficients, due to their properties,
we will assume \emph{without loss of generality} that
\begin{equation} \label{ecua:desigualdades-p-q-r}
\boxed{\ \phantom{\Bigm|} 2 \le r \le p \le q \le pr. \phantom{\Bigm|} \ }
\end{equation}
Note that under the assumption $p \le q$ condition (C1) in Proposition~\ref{prop:column-conditions}
is stronger than condition (R1) in Proposition~\ref{prop:row-conditions}.
\end{obse}

\section{Column-row cones} \label{sec:cr-cones} \hfill

For each triple $(p,q,r)$ of positive integers let $\crcone pqr$ be the set of all
$3$-dimen\-sional matrices of size $p \times q \times r$ with nonnegative real entries
satisfying the vanishing conditions and the column and row inequalities given in
the Definition~\ref{defi:cr-pol}.
It is a convex cone with vertex at the origin, which we call \emph{\textbf{column-row cone}}.
If $\la$, $\mu$ and $\tau$ are as in Section~\ref{sec:correspondencia}, then the column-row
polytope $\crpollibre \tau$ is the intersection of the cone $\crcone pqr$ with the
affine space $\mtrlibre \la \mu \tau$.

In this section we compute the dimension of the column-row cone and the highest dimension of a
column-row polytope.

\begin{obses} \label{obse:cr-prop}
(1) Since the column and row conditions are symmetric, the linear map
$(x_{i,j,k}) \asocia (x_{j,i,k})$ defines a linear isomorphism
$\crcone pqr \cong \crcone qpr$.
Thus we can assume without loss of generality that $p \le q$.
This is in concordance with~\eqref{ecua:desigualdades-p-q-r}.

(2) Let $( p^\prime, q^\prime, r^\prime )$ be another triple of positive integers such that
$p^\prime \ge p$, $q^\prime \ge q$ and $r^\prime \ge r$.
For any $X=(x_{i,j,k}) \in \crcone pqr$ let $E(X) = Y = (y_{i,j,k})$ be defined by
\begin{equation*}
y_{i,j,k} =
\begin{cases}
x_{i,j,k} & \text{if $i\le p$, $j \le q$ and $k\le r$};   \\
0 & \text{otherwise}.
\end{cases}
\end{equation*}
Then the correspondence $X \asocia E(X)$ defines an injective map
\begin{equation*}
\fun E{\crcone pqr}{\crcone {p^\prime} {q^\prime} {r^\prime}}
\end{equation*}
whose image is the face of $\crcone {p^\prime} {q^\prime} {r^\prime}$
given by the set of equalities $x_{i,j,k} = 0$ with $i >p$, $j > q$ and $k > r$.

(3) If $pr < q$, condition (C1) in Proposition~\ref{prop:column-conditions} implies that
the last column of $\matcol X$ is equal to $0$, for any $X \in \crcone pqr$.
Then, by the previous remark,
\begin{equation*}
\crcone pqr \cong \crcone p{q-1}r.
\end{equation*}
Thus we can assume without loss of generality that $q \le pr$.
This is in accordance with~\eqref{ecua:desigualdades-p-q-r}.
\end{obses}

As in Section~\ref{sec:correspondencia} we use the following notation
for $X = (x_{i,j,k}) \in \crcone pqr$
\begin{equation} \label{ecua:notacion-matriz-tres-d}
\entrada ijk = x_{i,j,k}, \quad \text{and} \quad X^{(k)} = \multimat x.
\end{equation}
The first level matrix $\nivel 1$ of a matrix $X\in \crcone pqr$ has a particularly nice form:

\begin{prop} \label{prop:primer-nivel}
If $p \le q$, the first level matrix of any $X \in \crcone pqr$ has the form:
\begin{equation} \label{ecua:primer-nivel}
\nivel 1 =
\begin{bmatrix}
x_1 & x_2 & \cdots & x_p & 0 & \cdots & 0 \\
x_2 & x_3 & \cdots & 0 & 0 & \cdots & 0 \\
\vdots & \vdots &  \ddots & \vdots & \vdots & \ddots & \vdots \\
x_p & 0 & \cdots & 0 & 0 & \cdots & 0
\end{bmatrix}.
\end{equation}
In particular, the entries of $\nivel 1$ in each diagonal $D_k = \{ (i,j) \mid i+j = k +1 \}$ coincide.
\end{prop}
\begin{proof}
This is proved in the same way as Corollary~\ref{coro:cano-cano} using
Propositions~\ref{prop:column-conditions} and~\ref{prop:row-conditions}.
\end{proof}

\begin{prop} \label{prop:dim-crcone}
If $p \le q \le pr$, the dimension of a column-row cone is:
\begin{equation*}
\textstyle
\dim \crcone pqr = pqr - \binom{p}{2} - \binom {q}{2}.
\end{equation*}
\end{prop}
\begin{proof}
We show first that the number of independent entries in any element $X$ in $\crcone pqr$
is at most $d = pqr - \binom{p}{2} - \binom {q}{2}$.
Since $p\le q$, (C1) implies (R1).
Let $\matcol X=(b_{i,j})$.
Condition (C1) and the inequality $pr \ge q$ imply that the $\binom{q}{2}$ entries
$b_{i,j}$ of $\matcol X$ satisfying $i+j > pr +1$ are zero.
Proposition~\ref{prop:primer-nivel} implies that in the first level of any
matrix in $\crcone pqr$ there are, besides the zero entries just mentioned,
$\binom {p}{2}$ entries dependent on other entries.
Therefore the number of independent entries in $X$ is at most $d$.
In other words, $\dim \crcone pqr \le d$.
We proceed to prove the other inequality.
It is enough to show that $\crcone pqr$ contains an open hypercube of dimension $d$.
For this, we choose $q$ non-empty open intervals in $\real$, $I_j = (a_j, b_j)$
such that $0 < a_q$, $b_1 <1$ and $b_{j+1} < a_j$ for all $j \in \nume{q-1}$.
Thus, any element in $I_j$ is greater than any element in $I_{j+1}$.
We then consider the hypercube $H$ formed by all matrices $X= \multimat x$
of size $p \times q \times r$ whose entries are of the form
\begin{equation*}
\entrada ijk =
\begin{cases}
0 & \text{if $i+j-1 > kp$}; \\
q(r-1) + \varepsilon_{i+j-1}^{(1)} & \text{if $i+j-1 \le p$ and $k=1$}; \\
\varepsilon_{i,j}^{(k)} & \text{if $i+j-1 \le kp$ and $k \ge 2$},
\end{cases}
\end{equation*}
where $\varepsilon_j^{(1)} \in I_j$ for $j \le p$ and $\varepsilon_{i,j}^{(k)} \in I_j$
if $i+j-1 \le kp$ and $k\ge 2$.
This hypercube has dimension $d$ (see Example~\ref{ejem:crcone-362}).
It is straightforward to check that each $X$ satisfies (C1), (C2) and (R1).
For (R2) we use that $1 + \varepsilon_i^{(1)} > \varepsilon_{i+1,q}^{(2)}$,
$1 + \varepsilon_{i,j+1}^{(k)} > \varepsilon_{i+1,j}^{(k)}$ and
$1 + \varepsilon_{i,1}^{(k)} > \varepsilon_{i+1,q}^{(k+1)}$.
That is the reason why we included the summand $q(r-1)$ in the definition of the nonzero
entries of $\nivel 1$.
Therefore $H$ is contained in $\crcone pqr$.
Thus $\dim \crcone pqr \ge d$ and the claim follows.
\end{proof}

\begin{ejem} \label{ejem:crcone-362}
Let $X$ be one of the matrices defined in the proof of Proposition~\ref{prop:dim-crcone} for the cone
$\crcone 362$.
Then its level matrices are
\begin{equation*}
\nivel 1 =
\begin{bmatrix}
6 + \varepsilon_1^{(1)} & 6 + \varepsilon_2^{(1)} & 6 + \varepsilon_3^{(1)} &
\hspace{.5em} 0 \hspace{.5em} & \hspace{.5em} 0 \hspace{.5em} & \hspace{.5em} 0  \hspace{.4em}\\[.1cm]
6 + \varepsilon_2^{(1)} & 6 + \varepsilon_3^{(1)} & 0 & 0 & 0 & 0 \\[.1cm]
6 + \varepsilon_3^{(1)} & 0 & 0 & 0 & 0 & 0
\end{bmatrix},
\end{equation*}
\begin{equation*}
X^{(2)} =
\begin{bmatrix}
\hspace{.8em} \varepsilon_{1,1}^{(2)} \hspace{.9em} & \hspace{.9em} \varepsilon_{1,2}^{(2)} \hspace{.9em} &
\hspace{.9em} \varepsilon_{1,3}^{(2)} \hspace{.9em} &
\varepsilon_{1,4}^{(2)} & \varepsilon_{1,5}^{(2)} & \varepsilon_{1,6}^{(2)} \\[.1cm]
\varepsilon_{2,1}^{(2)} & \varepsilon_{2,2}^{(2)} & \varepsilon_{2,3}^{(2)} &
\varepsilon_{2,4}^{(2)} & \varepsilon_{2,5}^{(2)} & 0 \\[.1cm]
\varepsilon_{3,1}^{(2)} & \varepsilon_{3,2}^{(2)} & \varepsilon_{3,3}^{(2)} &
\varepsilon_{3,4}^{(2)} & 0 & 0
\end{bmatrix}.
\end{equation*}
\end{ejem}

As an application we give a new proof of a result due to M.~Clausen
and H.~Meier (Satz~1.1 in~\cite{cm93}) and Y.~Dvir (Theorem~1.6 in~\cite{dvi93}).
We derive it from the following result.
Recall that the \emph{\textbf{intersection}} of two partitions $\la$ and $\mu$ is
\begin{equation*}
\la \cap \mu = ( \min\{\la_1, \mu_1\}, \min\{\la_2, \mu_2\}, \dots ).
\end{equation*}

\begin{teor} \label{teor:cdm}
Let $\la$ and $\mu$ be partitions of $n$.

(1) If $\nu$ is a partition of $n$ such that $\crpol \neq \vacio$, then $\nu_1 \le |\la \cap \mu|$.

(2) Let $\zeta = (k, 1^\ell)$ be a partition of $n$ such that $k \le |\la \cap \mu|$.
Then $\# \crpollibre \zeta >0$.
\end{teor}
\begin{proof}
(1) By hypothesis, there is an $X \in \crpol$.
Suppose $\la$ and $\mu$ have lengths $p$ and $q$, respectively.
As we noted before we may assume $p \le q$.
Since $\nivel 1$ is as in~\eqref{ecua:primer-nivel}, one has for all $i \in \nume p$
that
\begin{equation*}
\suma jip x_j \le \min \{ \la_i, \mu_i \}.
\end{equation*}
Hence,
\begin{equation*}
\nu_1 = \suma j1p j x_j = \suma i1p \suma jip x_j \le \suma i1p \min \{ \la_i, \mu_i \} = |\la \cap \mu|.
\end{equation*}

(2) We make induction on $\ell$.
If $\ell = 0$, $k=n$.
Therefore, $\la = \mu$.
It is easy to see that ${\rm CR}(\la,\la; (n))$ has exactly one point, which is integral.
Now suppose $\ell >0$.
Let $\rho$ be any partition of $k$ contained in $\la \cap \mu$.
Since $k<n$, there exists partitions $\alpha$ and $\beta$ of $n-1$ such that
$\rho \subseteq \alpha \subseteq \la$ and $\rho \subseteq \beta \subseteq \mu$.
Hence, there is $i \in \nume p$ such that $\alpha_i + 1 = \la_i$ and $\alpha_s = \la_s$
for all $s \neq i$, and there is $j \in \nume q$ such that $\beta_j + 1 = \mu_j$ and
$\beta_t = \mu_t$ for all $t \neq j$.
By induction hypothesis there is some integer point $Y \in {\rm CR}(\alpha, \beta; (k, 1^{\ell - 1})$.
It has level matrices $Y^{(1)}, \dots, Y^{(\ell)}$.
We define $Y^{(\ell + 1)}$ as the matrix with a $1$ is position $(i,j)$ and with $0$'s
in the remaining entries.
Let $X$ have level matrices $Y^{(1)}, \dots, Y^{(\ell)}, Y^{(\ell + 1)}$.
Then $X$ has $1$-marginals $\la$, $\mu$, $\zeta$.
We need to show that $X$ is in $\crpollibre \zeta$.
We know that $Y$ satisfies the column conditions given in (C2) in Proposition~\ref{prop:column-conditions}.
Since $X$ was obtained from $Y$ by adding one level matrix that has only one nonzero entry in
position $(i,j)$, we only have to check column inequalities between the $(j-1)$-th column and
the $j$-th column of $\matcol X$ and row inequalities between the $(i-1)$-th row and the $i$-th row
of $\matrow X$.
The number of terms on each side of the new column inequalities is at least $p \ell - j + 2$.
The sum of the terms on the left-hand side of any such inequality is always $\beta_{j-1}$
and the sum of the terms on the right-hand side of any such inequality is either $\beta_j$ or
$\beta_j + 1$.
Since $\beta_{j-1} = \mu_{j-1} \ge \mu_j = \beta_j + 1$, $X$ satisfies all column inequalities.
In a similar way one proves that $X$ satisfies all row inequalities given in (R2) in
Proposition~\ref{prop:row-conditions}.
Hence $X$ is in $\crpollibre \zeta$ and $\# \crpollibre \zeta >0$.
\end{proof}

\begin{coro}[Clausen-Meier, Dvir]
Let $\la$, $\mu$ be partitions of $n$.

(1) If $\nu$ is a partition of $n$ such that $\coefi >0$, then $\nu_1 \le |\la \cap \mu|$.

(2) Let $k = |\la \cap \mu|$.
Then there is a partition $\nu$ of $n$ such that $\coefi >0$ and $\nu_1 = |\la \cap \mu|$.
\end{coro}
\begin{proof}
(1) If $\coefi >0$, then, by Corollary~\ref{coro:kronecker bounds}, $\crpol$ is not empty.
Part (1) of Theorem~\ref{teor:cdm} implies the result.

(2) Let $k = |\la \cap \mu|$ and $\zeta = (k, 1^{n-k})$.
By part (2) in Theorem~\ref{teor:cdm} one has $\# \crpollibre \zeta >0$.
It follows from Young's rule and Theorem~\ref{teor:cr-poli} that
\begin{equation} \label{ecua:young-escuadra}
\# \crpollibre \zeta = \sum_{\nu \vdash n} K_{\nu\zeta}\, \coefi.
\end{equation}
By part (1) and a property of the Kostka numbers the positive terms in the sum correspond to partitions
$\nu$ such that $\nu_1 \le k$ and $\nu \mayori \zeta$ in the dominance order.
Hence $\nu_1 = k$.
Since the sum in equation~\eqref{ecua:young-escuadra} is positive, there is at least one partition
$\nu$ of $n$ such that $\coefi >0$ and $\nu_1 = |\la \cap \mu|$.
\end{proof}

\section{Dimension of the column-row polytope} \label{sec:dim column-row polytope} \hfill

The dimensions of the column-row polytopes are given by the following result.

\begin{teor} \label{teor:dim-crpolytope}
Let $\la$, $\mu$ be partitions of some integer $n$, let $\tau$ be a composition of $n$,
let $p$, $q$, $r$ be the lengths of $\la$, $\mu$, $\tau$, respectively.
If $2 \le r$ and $p \le q \le pr$, the polytope $\crpollibre \tau$ has dimension at most
\begin{equation} \label{ecua:dim-cr-polytope}
\textstyle
pqr - \binom{p}{2} - \binom {q}{2} - (p+q+r) + 2.
\end{equation}
\end{teor}
\begin{proof}
Let $X = \multimat x$ be a matrix in $\crpollibre \tau$.
Since the sum of the entries in the first level of $X$ is $\tau_1$ and since $p \le q$,
Proposition~\ref{prop:primer-nivel} implies that there are $p-1$ degrees of freedom
in $\nivel 1$.
Since $q \le pr$, the number of zeros in $X^{(2)}, \dots , X^{(r)}$ forced by
condition (C1) in Proposition~\ref{prop:column-conditions} is $\binom{q-p}{2}$.
The entries of $X^{(k)}$, $1 < k < r$, sum $\tau_k$, therefore in this level there are
$pq - 1$ degrees of freedom minus the zeros forced by condition (C1).
In $X^{(r)}$ we have to take into account the plane sums $\la$ and $\mu$.
Therefore, in this level there are $(p-1)(q-1)$ degrees of freedom minus the zeros forced
by condition (C1).
Here we are using that $r \ge 2$.
As an illustration see Example~\ref{ejem:grados-de-libertad-cr-363}.
Hence, the total number of degrees of freedom is
\begin{equation*}
\textstyle
(p-1) + (pq-1)(r-2) + (p-1)(q-1) - \binom{q-p}{2}.
\end{equation*}
This number is the highest possible dimension of $\crpollibre \tau$ and it is easily seen to be
equal to~\eqref{ecua:dim-cr-polytope}.
\end{proof}

\begin{obses} \label{obses:condiciones en parametros de CR-2}
(1) The polytope $\crpollibre {(n)}$ has either exactly one point if $\la = \mu$ or
is empty if $\la \neq \mu$.
This follows either from Proposition~\ref{prop:primer-nivel} or from the definition
of $\liri \la\mu\rho$ in~\eqref{ecua:liri tableau}.
Since the case $r=1$ is a trivial one, we may assume $r \ge 2$.

(2) Assuming the conditions of Theorem~\ref{teor:dim-crpolytope} on $p$, $q$ and $r$, one sees
that the dimension given in~\eqref{ecua:dim-cr-polytope} is the smallest when we choose $r$
to satisfy the inequality $r \le p$.
In order to have a more efficient implementation of Barvinok's algorithm
we want to have the column row polytopes $\crpollibre \gamma$ in~\eqref{ecua:kron-alternante}
of dimension as small as possible.
Thus, the geometry of these polytopes also leads us to assume the
inequalities~\eqref{ecua:desigualdades-p-q-r}, which were derived from
representation theoretic considerations.
\end{obses}

\begin{ejem} \label{ejem:grados-de-libertad-cr-363}
For example, for $p=3$, $q=6$ and $r=3$ we indicate below with an $f$ the entries
that take arbitrary values and with $d$ those who are determined by the free
variables.
According to the formula we have 26 degrees of freedom:
\begin{equation*}
\begin{bmatrix}
d & f & f & 0 & 0 & 0 \\
d & d & 0 & 0 & 0 & 0 \\
d & 0 & 0 & 0 & 0 & 0
\end{bmatrix}
\qquad
\begin{bmatrix}
d & f & f & f & f & f \\
f & f & f & f &f &  0 \\
f & f & f & f & 0 & 0
\end{bmatrix}
\qquad
\begin{bmatrix}
d & d & d & d & d & d \\
d & f & f & f & f & f \\
d & f & f & f & f & f
\end{bmatrix}.
\end{equation*}
\end{ejem}

\section{Some column and row inequalities} \label{sec:column and row ineqs} \hfill

In this section we introduce an alternative way of presenting column and row inequalities.
This will help us to determine all the facets of the column-row cone in~\cite{vakron2}.
Here we restrict ourselves to the inequalities needed in Section~\ref{sec:red of dim}.
Given an element $X \in \crcone pqr$ we describe the inequalities corresponding to the facets
in terms of the level matrices $X^{(k)}$ of $X$ as defined in~\eqref{ecua:notacion-matriz-tres-d}.
Because of Remarks~\ref{obse:cr-prop}.1, \ref{obse:cr-prop}.3, \ref{obses:condiciones en parametros de CR-2}.1
and~\ref{obses:condiciones en parametros de CR-2}.2 we will assume inequalities~\eqref{ecua:desigualdades-p-q-r},
whenever they are needed.

We use throughout the convention that $\suma iab v_i = 0$ if $a>b$.

\subsection{Column inequalities} \label{subsec:column-ineq} \hfill

We consider the column inequalities
presented in part (C2) of~Proposition~\ref{prop:column-conditions}.
Due to condition (C1) in the same proposition and to the very particular
shape of $\nivel 1$ shown in Proposition~\ref{prop:primer-nivel}
several of these inequalities will reduce, after cancelations and the elimination
of entries forced to be zero, to one and the same inequality.
We give here a smaller set of column inequalities that is equivalent to the original
set given in~(C2).

Let $X \in \crcone pqr$ and look at $\matcol X$ as defined in~\eqref{ecua:definicion de A-col}.
Since the first level matrix $\nivel 1$ of $X$ has the very simple description
given in~\eqref{ecua:primer-nivel}, we write
\begin{equation*}
\text{ $x_j = \entrada 1j1 = \entrada j11$ for $j \in \nume p$ and
$x_j = \entrada 1j1 = 0$ for $j>p$.}
\end{equation*}
Instead of working with the whole $\nivel 1$, we work only with its first row.
In other words, we remove the last $p-1$ rows of $\matcol X$ and work with the resulting matrix,
denoted $\matcol {\ol X}$, of size $\bigl( p(r-1) + 1 \bigr) \times q$.
This smaller matrix contains the same information as $\matcol X$ and is more appropriate for
handling column inequalities.
As an illustration consider the column inequality
\begin{equation*}
\entrada p11 + \cdots + \entrada 211 + \entrada 111 \ge \entrada {p-1}21 + \cdots + \entrada 121 + \entrada p22.
\end{equation*}
Since $\entrada {p-i}11 = \entrada {p-i-1}21$, it is equivalent to $x_1 \ge \entrada p22$.
For this reason it is enough to consider only column inequalities that involve
exactly one entry of the first row of $\nivel 1$ on its left-hand side.
We present in Table~\ref{tabla:ejemplo de elemento en CR} a generic element of $\matcol {\ol X}$
for $p =4$, $q = 11$ and $r = 3$.
Also, the entries forced to be zero coming from the conditions (C1)
in~Proposition~\ref{prop:column-conditions} make several inequalities redundant.
Column inequalities involving this kind of entries will be considered in~\cite{vakron2}.

\begin{table}
\footnotesize
\begin{center}
\renewcommand{\arraystretch}{1.7}
\begin{tabular}{| c c c c c c c c c c c|}
\hline
$\entrada 113$ & $\entrada 123$ & $\entrada 133$ & $\entrada 143$ &
$\entrada 153$ & $\entrada 163$ & $\entrada 173$ & $\entrada 183$ &
$\entrada 193$ & $\entrada 1{10}3$ & $\entrada 1{11}3$ \\
$\entrada 213$ & $\entrada 223$ & $\entrada 233$ & $\entrada 243$ &
$\entrada 253$ & $\entrada 263$ & $\entrada 273$ & $\entrada 283$ &
$\entrada 293$ & $\entrada 2{10}3$ & $\entrada 2{11}3$ \\
$\entrada 313$ & $\entrada 323$ & $\entrada 333$ & $\entrada 343$ &
$\entrada 353$ & $\entrada 363$ & $\entrada 373$ & $\entrada 383$ &
$\entrada 393$ & $\entrada 3{10}3$ & $0$ \\
$\entrada 413$ & $\entrada 423$ & $\entrada 433$ & $\entrada 443$ &
$\entrada 453$ & $\entrada 463$ & $\entrada 473$ & $\entrada 483$ &
$\entrada 493$ & $0$ & $0$ \\
\hline
$\entrada 112$ & $\entrada 122$ & $\entrada 132$ & $\entrada 142$ & $\entrada 152$ &
$\entrada 162$ & $\entrada 172$ & $\entrada 182$ & $0$ & $0$ & $0$ \\
$\entrada 212$ & $\entrada 222$ & $\entrada 232$ & $\entrada 242$ & $\entrada 252$ &
$\entrada 262$ & $\entrada 272$ & $0$ & $0$ & $0$ & $0$ \\
$\entrada 312$ & $\entrada 322$ & $\entrada 332$ & $\entrada 342$ & $\entrada 352$ &
$\entrada 362$ & $0$ & $0$ & $0$ & $0$ & $0$ \\
$\entrada 412$ & $\entrada 422$ & $\entrada 432$ & $\entrada 442$ & $\entrada 452$ &
$0$ & $0$ & $0$ & $0$ & $0$ & $0$\\
\hline
$x_1$ & $x_2$ & $x_3$ & $x_4$ & $0$ & $0$ & $0$ & $0$ & $0$ & $0$ & $0$ \\
\hline
\end{tabular}
\end{center}
\caption{A generic element of $\matcol {\ol X}$ when $r = 3$, $p =4$ and $q = 11$} \label{tabla:ejemplo de elemento en CR}
\end{table}

We relate the numbers $u$, $v$, and $w$ of an entry $\entrada uvw$ of $\matcol {\ol X}$ with
its $t$-th place in  the $v$-th column counting from bottom to top and starting in $\entrada pv2$.
For $t \in \nume {p(r-1)}$ let us write
\begin{equation*}
\text{$t = c_t p + d_t$ where $c_t \ge 0$ and $d_t \in \nume p$.}
\end{equation*}
Then, the $t$-th entry from bottom to top in the $v$-th column of $\matcol {\ol X}$
starting at $\entrada pv2$ is $\entrada {p+1-d_t} v {c_t + 2}$.
For example, the $4$-th entry in the $v$-th column of $\matcol {\ol X}$ (with the values of
$p$, $q$ and $r$ given in Table~\ref{tabla:ejemplo de elemento en CR}) is
$\entrada 1v2$, and the $7$-th entry in the $v$-th column of $\matcol {\ol X}$ is $\entrada 2v3$.

Each column inequality has the form $\suma i1t y_i \ge \suma i1t z_i$, where the
$y_i$'s are entries in some column of $\matcol {\ol X}$, let's say in $j$-th column, and
the $z_i$'s are entries in $(j+1)$-th column of $\matcol {\ol X}$.
We proceed to describe these sums using the notation introduced in~\eqref{ecua:notacion-matriz-tres-d}
and~\eqref{ecua:primer-nivel}.
For each $j \in \nume q$ and each $k \in \nume{2, r}$ let
\begin{equation*}
\gamma_j^{(k)} = \suma i1p x_{i,j}^{(k)}
\end{equation*}
be the sum of all entries in the $j$-th column of $X^{(k)}$.
For $j \in \nume {p+1}$ and $t \in \nume{ (r-1)p }$ let
\begin{equation*}
S_{j,t}^c = \suma k1{c_t} \gamma_j^{(k+1)} + \suma i{p+1-d_t}p \entrada ij{c_t+2}.
\end{equation*}
This is the sum of the first $t$ elements of the $j$-th column of $\matcol {\ol X}$
counted from bottom to top, starting at $\entrada pj2$ and finishing at
$\entrada {p+1-d_t} j {c_t + 2}$.
We also define $S_{j,0}^c = 0$.
For example, if $j=4$ and $t = 6$, we have (with the values of $p$, $q$ and $r$
given in Table~\ref{tabla:ejemplo de elemento en CR})
that $S_{4,6}^c = \gamma_4^{(2)} + \entrada 443 + \entrada 343$.
This sum starts at $\entrada 442$, ends at $\entrada 343$ and has $6$ summands.

\begin{nota}
For each $j \in \nume p$ and $t \in \nume{p(r-1)}$ we denote by $\negra{\colineq jt}$
the column inequality
\begin{equation} \label{ecua:desigualdad Cjt}
x_j + S_{j,t-1}^c \ge S_{j+1,t}^c.
\end{equation}
It compares a sum of $t$ terms in the $j$-th column of $\matcol {\ol X}$ with a sum of $t$ terms
in the $(j+1)$-th column of $\matcol {\ol X}$.
Note that $\colineq pt$ is only defined when $p<q$.
Column inequalities whose left-hand side contains entries in $j$-th column for
$j>p$ will be considered in~\cite{vakron2}.
\end{nota}

\begin{ejem} \label{ejem:col ineq-figura 1}
Some examples of inequalities with the values of $p$, $q$ and $r$ given in
Table~\ref{tabla:ejemplo de elemento en CR} are:
\begin{align*}
\colineq 21: & \quad x_2 \ge \entrada 432. \\
\colineq 46: & \quad x_4 + \gamma_4^{(2)} + \entrada 443 \ge \gamma_5^{(2)} + \entrada 453 + \entrada 353.
\end{align*}
\end{ejem}

\subsection{Row inequalities} \label{subsec:row-ineq} \hfill

Let $X \in \crcone pqr$.
We use the notation of Subsection~\ref{subsec:column-ineq}.
For the description of the row inequalities we start with $\matrow X$ as defined
in~\eqref{ecua:definicion de A-row}, delete its last $q-1$ columns and denote the
resulting matrix of size $\bigl( p \times (r-1)q + 1 \bigl)$ by $\matrow {\ol X}$.
This smaller matrix contains the same information as $X$ but it is more appropriate
for dealing with row inequalities.
We present in Table~\ref{tabla:ejemplo X-row de elemento en CR} a generic element of $\matrow {\ol X}$
for  $r = 3$, $p =3$ and $q = 8$.
Note that, when $p<q$ the structure of row inequalities differs from the structure of column inequalities,
because the entries forced to be zero in $\matrow {\ol X}$ can be separated by nonzero entries.
For example, in Table~\ref{tabla:ejemplo X-row de elemento en CR}, $\entrada 383$ is forced to be zero,
but next to its right $\entrada 312$ could be different from zero and then some entries further to the
right $\entrada 352$ is again forced to be zero.
So, the string of zeros in some rows of $\matrow {\ol X}$ could be not connected.

\begin{table}
\scriptsize
\begin{center}
\renewcommand{\arraystretch}{1.7}
\begin{tabular}{| c c c c c c c c | c c c c c c c c | c |}
\cline{1-17}
$\entrada 113$ & $\entrada 123$ & $\entrada 133$ & $\entrada 143$ & $\entrada 153$ &
$\entrada 163$ & $\entrada 173$ & $\entrada 183$ &
$\entrada 112$ & $\entrada 122$ & $\entrada 132$ & $\entrada 142$ & $\entrada 152$ &
$\entrada 162$ & $0$ & $0$ & $x_1$ \\
$\entrada 213$ & $\entrada 223$ & $\entrada 233$ & $\entrada 243$ & $\entrada 253$ &
$\entrada 263$ & $\entrada 273$ & $\entrada 283$ &
$\entrada 212$ & $\entrada 222$ & $\entrada 232$ & $\entrada 242$ & $\entrada 252$ &
$0$ & $0$ & $0$ & $x_2$ \\
$\entrada 313$ & $\entrada 323$ & $\entrada 333$ & $\entrada 343$ & $\entrada 353$ &
$\entrada 363$ & $\entrada 373$ & $0$ &
$\entrada 312$ & $\entrada 322$ & $\entrada 332$ & $\entrada 342$ & $0$ &
$0$ & $0$ & $0$ & $x_3$ \\
\cline{1-17}
\end{tabular}
\end{center}
\caption{A generic element of $\matrow {\ol X}$ when $r = 3$, $p =3$ and $q = 8$} \label{tabla:ejemplo X-row de elemento en CR}
\end{table}

We now count the entries in the $u$-th row of $\matrow {\ol X}$ from right to left starting
with $\entrada uq2$.
For $s \in \nume{q(r-1)}$ let us write
\begin{equation*}
\text{$s = e_sq + f_s$ where $e_s \ge 0$ and $f_s \in \nume q$}.
\end{equation*}
Then, the $s$-th entry from right to left in the $u$-th row of $\matrow {\ol X}$
starting at $\entrada uq2$ is $\entrada u {q+1-f_s} {e_s+2}$.

Each row inequality has the form $\suma k1s y_k \ge \suma k1s z_k$, where the
$y_k$'s are entries in some row of $\matrow {\ol X}$, let's say in $i$-th row,
and the $z_k$'s are entries in $(i+1)$-th row of $\matrow {\ol X}$.
For each $i \in \nume p$ and each $k \in \nume{2, r}$ let
\begin{equation*}
\rho_i^{(k)} = \suma j1q x_{i,j}^{(k)}
\end{equation*}
be the the sum of all entries in the $i$-th row of $X^{(k)}$.
For $i \in \nume p$ and $s \in \nume{ (r-1)q }$ let
\begin{equation*}
S_{i,s}^r = \suma k1{e_s} \rho_i^{(k+1)} + \suma j{q+1-f_s}q \entrada ij{e_s+2}.
\end{equation*}
This is the sum of the first $s$ elements of the $i$-th row of $\matrow {\ol X}$
counted from right to left, starting at $\entrada iq2$ and finishing at
$\entrada i {q+1-f_s} {e_s+2}$.
We also define $S_{i,0}^r = 0$.
For example, if $i=2$ and $s=11$, we have (with the values of $p$, $q$ and $r$ given
in Table~\ref{tabla:ejemplo X-row de elemento en CR}) that
$S_{2,11}^r = \rho_2^{(2)} + \entrada 283 + \entrada 273 + \entrada 263$.
This sum starts at $\entrada 282$, ends at $\entrada 263$ and has $11$ summands.

\begin{nota}
For each $i \in \nume {p-1}$ and $s \in \nume{(r-1)q}$ we denote by $\negra{\rowineq is}$
the row inequality
\begin{equation} \label{ecua:desigualdad Ris}
x_i + S_{i,s-1}^r \ge S_{i+1,s}^r.
\end{equation}
It compares a sum of $s$ terms of the $i$-th row of $\matrow {\ol X}$ with a sum of $s$ terms of the
$(i+1)$-th row of $\matrow {\ol X}$.
The left-hand side of $\rowineq is$ starts at $x_i$, but after $x_i$
some entries might be forced to be zero (as it is in the case depicted in
Table~\ref{tabla:ejemplo X-row de elemento en CR}).
The right-hand side starts at $\entrada {i+1}q2$, but the first entry not forced to be zero is
$\entrada {i+1}{ \min \{2p-i, q \} }2$.
Hence, among the row inequalities just defined there might be (if $p<q$) several redundant
inequalities.
We determine the irredundant row inequalities in~\cite{vakron2}.
\end{nota}

\begin{ejem} \label{ejem:row ineq-figura 2}
Some examples of inequalities with the values of $p$, $q$ and $r$ given
in Table~\ref{tabla:ejemplo X-row de elemento en CR}.
\begin{align*}
\rowineq 21:  & \quad x_2 \ge \entrada 382, \\
\rowineq 18:  & \textstyle \quad x_1 + \suma j28 \entrada 1j2 \ge \suma j18 \entrada 2j2 = \rho_2^{(2)}, \\
\rowineq 2 {11}: & \textstyle \quad  x_2 + \rho_2^{(2)} + \suma j78 \entrada 2j3 \ge
\rho_3^{(2)} + \suma j68 \entrada 3j3.
\end{align*}
\end{ejem}

\section{Reduction of dimension} \label{sec:red of dim} \hfill

Let $\la$, $\mu$ and $\nu$ be partitions of $n$ of lengths $p$, $q$, $r$, respectively.
As we mentioned in the introduction V.~Baldoni and M.~Vergne showed in~\cite{bavew18}
that the Kronecker coefficient $\coefili {t\la} {t\mu} {t\nu}$ is a quasi-polynomial function
of $t$ of degree at most
\begin{equation*}
\textstyle
d = pqr - \binom p2 - \binom q2 - \binom r2 - (p+q+r) + 2.
\end{equation*}
Hence, one would like to express $\coefi$ as a sign-alternating sum of numbers of integer points in
some polytopes whose definition depends on $\la$, $\mu$ and $\nu$ and
whose dimension does not exceed $d$.
It follows from identities~\eqref{ecua:kron-alternante} and~\eqref{ecua:dim-cr-polytope}
that every Kronecker coefficient can be computed in such a way with the constraint that the
dimension of the required polytopes does not exceed
\begin{equation*}
\textstyle
e = pqr - \binom p2 - \binom q2 - (p+q+r) + 2.
\end{equation*}
Since $e-d = \binom r2$, identity~\eqref{ecua:dim-cr-polytope} is not optimal.

In Theorem~\ref{teor:kron-eficiente} we show how to express $\coefi$ as a
sign-alternating sum of numbers of integer points of polytopes
whose dimension is at most $e-1$, hence reaching the optimal dimension when $r=2$.
We expect that a similar method will help to reach the optimal dimension when $r=3$.

Assume the inequalities given in~\eqref{ecua:desigualdades-p-q-r} and
recall equation~\eqref{ecua:kron-alternante}:
\begin{equation*}
\coefi = \sum_{\gamma \mayori \nu} K_{\gamma\nu}^{(-1)}\, \# \crpollibre \gamma.
\end{equation*}
This summation comes from the expansion of a Jacobi-Trudi determinant
in the ring of symmetric functions (see Section~\ref{sec:coef-de-kron}):
\begin{equation*}
s_\nu = \det\left( h_{\nu_i -i + j} \right)_{i,j \in \nume r}.
\end{equation*}
If we expand this determinant by cofactors along the first column, and then continue successively
expanding the remaining determinants along the second up to the $(r-2)$-th column, we get an
expression of the form
\begin{equation*}
s_\nu = \sum \pm h_\rho \cdot \left[ h_{(a,b)} - h_{(a+1,b-1)} \right] ,
\end{equation*}
where the various $\rho$'s are obtained from the successive cofactor expansions of the first $r-2$ columns,
and $a$ and $b$ are numbers appearing in the remaining $2\times 2$ determinants
\begin{equation*}
\det
\begin{bmatrix}
h_a & h_{a+1} \\
h_{b-1} & h_b
\end{bmatrix}
= h_{(a,b)} - h_{(a+1,b-1)},
\end{equation*}
where $a$ and $b$ are nonnegative integers.
For example, if $\nu = (7,4,2,1)$, after permuting coordinates of the various compositions,
the determinant can be expressed in the following way
\begin{multline*}
\det\left[
\begin{matrix}
h_7 & h_8 & h_9 & h_{10} \\
h_3 & h_4 & h_5 & h_6 \\
h_0 & h_1 & h_2 & h_3 \\
0        & 0        & h_0 & h_1
\end{matrix}
\right] = \\
  \left[ h_{(2,1,7,4)} - h_{(3,0,7,4)} \right]
- \left[ h_{(5,1,7,1)} - h_{(6,0,7,1)} \right]
- \left[ h_{(2,1,8,3)} - h_{(3,0,8,3)} \right] \\
+ \left[ h_{(9,1,3,1)} - h_{(10,0,3,1)} \right]
+ \left[ h_{(5,1,8)} - h_{(6,0,8)} \right]
- \left[ h_{(9,1,4)} - h_{(10,0,4)} \right].
\end{multline*}
In this way, equation~\eqref{ecua:kron-alternante} can be expressed in the form
\begin{equation} \label{ecua:kron-con-politopos}
\coefi = \sum_\tau \pm \left[ \# \crpollibre \tau - \# \crpollibre{\overline\tau} \right],
\end{equation}
where $\tau = (a, b, \rho_1, \dots, \rho_s)$ and $\ol\tau = (a+1, b-1, \rho_1, \dots, \rho_s)$
are of compositions of $n$.
This is possible because of Remark~\ref{obse:teor-kron-alternante}.
It follows from Lemma~4.8 in~\cite{va14} that each term
$\left[ h_{(a,b)} - h_{(a+1,b-1)} \right] \cdot h_\rho$ corresponds to the character of a
representation.
Therefore,
\begin{equation*}
\liri \la\mu \tau \ge \liri \la\mu {\ol\tau}.
\end{equation*}
Hence, by Theorem~\ref{teor:cr-poli}, one has that
\begin{equation*}
\# \crpollibre \tau - \# \crpollibre {\ol\tau} \ge 0.
\end{equation*}
In order to compute each difference
\begin{equation*}
\# \crpollibre \tau - \# \crpollibre {\ol \tau}
\end{equation*}
we define an injective map from a subset of $\crpollibre {\ol \tau}$ into $\crpollibre \tau$.
Recall from Section~\ref{sec:correspondencia} that $\transptres \tau$ denotes the affine
space of all matrices with real entries having 1-marginals $\la$, $\mu$, $\tau$.
For each pair $(\tau, \ol\tau)$ and each $\ell \in \nume p$ we define an affine isomorphism
\begin{equation} \label{ecua:cancelacion parcial}
\fun {\Phi_\ell} {\transptres {\ol\tau}} {\transptres \tau}
\end{equation}
as follows:
let $Z_\ell= \bigl(\multient zijk \bigr)$ be the matrix of size $p \times q \times r$ defined by
\begin{equation*}
\multient zijk =
\begin{cases}
\phantom{-}1 & \text{if $k=1$ and $i+j = \ell$;}\\
-1 & \text{if $k=1$ and $i+j = \ell + 1$;}\\
\phantom{-}1 & \text{if $k=2$, $i=\ell$ and $j=\ell$;}\\
\phantom{-}0 & \text{otherwise.}
\end{cases}
\end{equation*}
For example, if $p=2$, $q=3$ and $r=2$
\begin{equation*}
Z_1 =
\begin{bmatrix}
-1 & 0 & 0\phantom{-} \\ \phantom{-}0 & 0 & 0\phantom{-}
\end{bmatrix}
\quad
\begin{bmatrix}
\phantom{-}1 & 0 & 0\phantom{-} \\\phantom{-} 0 & 0 & 0\phantom{-}
\end{bmatrix}
\end{equation*}
and if $p=3$, $q=4$ and $r=3$
\begin{equation*}
Z_3 =
\begin{bmatrix}
\phantom{-}0 & \phantom{-}1 & -1 & \phantom{-}0\phantom{-} \\
\phantom{-}1 & -1 & \phantom{-}0 & \phantom{-}0\phantom{-} \\
-1 & \phantom{-}0 & \phantom{-}0 & \phantom{-}0\phantom{-}
\end{bmatrix}
\quad
\begin{bmatrix}
\phantom{-}0 & \phantom{-}0 & \phantom{-}0 & \phantom{-}0\phantom{-} \\
\phantom{-}0 & \phantom{-}0 & \phantom{-}0 & \phantom{-}0\phantom{-} \\
\phantom{-}0 & \phantom{-}0 & \phantom{-}1 & \phantom{-}0\phantom{-}
\end{bmatrix}
\quad
\begin{bmatrix}
\phantom{-}0 & \phantom{-}0 & \phantom{-}0 & \phantom{-}0\phantom{-} \\
\phantom{-}0 & \phantom{-}0 & \phantom{-}0 & \phantom{-}0\phantom{-} \\
\phantom{-}0 & \phantom{-}0 & \phantom{-}0 & \phantom{-}0\phantom{-}
\end{bmatrix}.
\end{equation*}
Then, $\Phi_\ell$ is defined by $\Phi_\ell(X) = X + Z_\ell$.
This isomorphism and its inverse map integer points to integer points.

The proofs of the next two lemmas are straightforward.

\begin{lema} \label{lema:cara-menos}
Let $X = \bigl( \entrada ijk \bigr) \in \crpollibre {\ol\tau}$.
Then

(1) if $\ell = p$ and $p = q$, $\Phi_\ell(X)$ is in $\crpollibre \tau$ if and only if $x_p \ge 1$.

(2) If $\ell < p$ or $p < q$, $\Phi_\ell(X)$ is in $\crpollibre \tau$ if and only if
\begin{equation} \label{lema:caras de columnas a restar}
x_\ell + \suma i{p-t+2}p \entrada i \ell 2 \ge 1 + \suma i{p-t+1}p \entrada i {\ell + 1} 2
\end{equation}
for all $t \in \nume{p - \ell + 1}$ and
\begin{equation} \label{lema:caras de renglones a restar}
x_\ell + \suma j{q-s+2} q \entrada \ell j 2 \ge 1 + \suma j {q-s+1}q \entrada {\ell + 1} j 2
\end{equation}
for all $s \in \nume{q-\ell + 1}$.
\end{lema}

\begin{obses}
(1) If $\ell = p$ and $p<q$, the inequalities~\eqref{lema:caras de renglones a restar} are not defined.
Hence, for the equivalence, one only considers the inequalities~\eqref{lema:caras de columnas a restar}.

(2) If $\ell < p$ and $p<q$, due to the entries of $X$ forced to be zero (indicated with $0$'s
in Table~\ref{tabla:ejemplo de elemento en CR}), some of the inequalities corresponding
to~\eqref{lema:caras de renglones a restar} may be redundant.

(3) The inequality~\eqref{lema:caras de columnas a restar} is related
to the inequality $\colineq \ell t$ as defined in~\eqref{ecua:desigualdad Cjt}
and the inequality~\eqref{lema:caras de renglones a restar}
is related to the inequality $\rowineq \ell s$ as defined in~\eqref{ecua:desigualdad Ris}.
\end{obses}

\begin{lema} \label{lema:cara-mas}
Let $X = \bigl( \entrada ijk \bigr) \in \crpollibre \tau$.
Then $X$ is in the image of $\Phi_\ell$ if and only if $x_{\ell-1} \ge 1$, $\multient x\ell \ell 2 \ge 1$,
\begin{equation} \label{lema:caras de columnas a sumar}
x_{\ell-1} + \suma i {p-t+2} p \multient x i {\ell-1} 2 \ge 1 + \suma i {p-t+1} p \multient x i \ell 2
\end{equation}
for all $t \in \nume {p-\ell}$ and
\begin{equation} \label{lema:caras de renglones a sumar}
x_{\ell-1} + \suma j {q-s+2} q \multient x {\ell-1} j 2 \ge 1 + \suma j {q-s+1} q \multient x \ell j 2
\end{equation}
for all $s \in \nume{q - \ell}$.
\end{lema}

\begin{obses}
(1) If $\ell = 1$, $X$ is in the image of $\Phi_\ell$ if and only if
$\multient x\ell \ell 2 \ge 1$.

(2) If $p<q$, due to the entries of $X$ forced to be zero (indicated with $0$'s
in Table~\ref{tabla:ejemplo de elemento en CR}), some of the inequalities corresponding
to~\eqref{lema:caras de renglones a sumar} may be redundant.

(3) The inequality~\eqref{lema:caras de columnas a sumar} is related
to the inequality $\colineq {\ell-1} t$, as defined in in~\eqref{ecua:desigualdad Cjt},
and the inequality~\eqref{lema:caras de renglones a sumar}
is related to the inequality $\rowineq {\ell-1} s$, as defined in~\eqref{ecua:desigualdad Ris}.
\end{obses}

\begin{nota}
We introduce the following notation for the (possibly empty) faces of $\crpollibre \tau$.
For each $i$, $j \in \nume p$ we denote
\begin{equation*}
\posiuno i \tau = \{ X = \bigl( \entrada ijk \bigr) \in \crpollibre \tau \mid x_i = 0 \}
\end{equation*}
and
\begin{equation*}
\posidos ij2\tau = \{ X = \bigl( \entrada ijk \bigr) \in \crpollibre \tau \mid \multient xij2 = 0 \}.
\end{equation*}
Also, let

(1) $\policol jt\tau$ be the intersection of $\crpollibre \tau$ with the hyperplane
associated to the inequality $\colineq jt$ for each $j \in \nume p$ and $t \in \nume p$.

(2) $\polirow is\tau$ be the intersection of $\crpollibre \tau$ with the hyperplane
associated to the inequality $\rowineq is$ for each $i \in \nume {p-1}$ and $s \in \nume q$.

Finally, we consider the following unions of faces:
\begin{align*}
F^+_\ell(\la, \mu, \tau) & = P_{\ell-1}(\la, \mu, \tau) \cup P_{\ell,\ell}^{(2)}(\la, \mu, \tau) \cup
\bigcup_{t=1}^{p-\ell} \policol {\ell-1}t\tau \cup
\bigcup_{s=1}^{q - \ell} \polirow {\ell-1}s\tau,\\
F^-_\ell(\la, \mu, \ol\tau) & = P_p(\la, \mu, \ol\tau) \quad\text{ if $\ell = p$ and $p=q$, and} \\
F^-_\ell(\la, \mu, \ol\tau) & =
\bigcup_{t=1}^{p-\ell+1} \policol \ell t {\ol\tau} \cup
\bigcup_{s=1}^{q - \ell + 1} \polirow \ell s {\ol\tau} \quad\text{ if $\ell < p$ or $p<q$.}
\end{align*}

Note that in the definitions some of the faces in the unions may be not defined (as
noted in the lemmas) or may be empty.
\end{nota}

Let $\la$, $\mu$, $\tau$, $\ol\tau$ and $\ell$ be as above.
We have defined an affine isomorphism $\Phi_\ell$ in~\eqref{ecua:cancelacion parcial}.
The number of integer points in $\crpollibre {\ol\tau}$ whose image is not
in $\crpollibre \tau$ is, according to Lemma~\ref{lema:cara-menos}, $\# F^-_\ell(\la, \mu, \ol\tau)$
and the number of integer points in $\crpollibre \tau$ that are not in the image of $\Phi_\ell$
is, according to Lemma~\ref{lema:cara-mas}, $\# F^+_\ell(\la, \mu, \tau)$.
The remaining integer points in the difference
$\# \crpollibre \tau - \# \crpollibre{\overline\tau}$ cancel out.
Hence, we have the following

\begin{prop} \label{prop:cancelacion geometrica}
Let $\la$, $\mu$, $\tau$, $\ol\tau$ be as above.
Then
\begin{equation} \label{ecua:cancelacion geometrica}
\# \crpollibre \tau - \# \crpollibre{\ol\tau} =
\# F^+_\ell(\la, \mu, \tau) - \# F^-_\ell(\la, \mu, \ol\tau).
\end{equation}
\end{prop}

From Theorem~\ref{teor:kron-alternante}, identity~\eqref{ecua:kron-con-politopos}
and Proposition~\ref{prop:cancelacion geometrica} we get the following result.

\begin{teor} \label{teor:kron-eficiente}
Let $\la$, $\mu$, $\nu$ be partitions of an integer $n$ and lengths $p$, $q$, $r$,
respectively.
Then, for each $\ell \in \nume p$ one has
\begin{equation} \label{ecua:crdif-pq}
\coefi = \sum_\tau \pm \left[ \# F^+_\ell(\la, \mu, \tau) - \# F^-_\ell(\la, \mu, \ol\tau) \right],
\end{equation}
and $\tau$ runs as in~\eqref{ecua:kron-con-politopos}.
Moreover, each of the summands above involves counting integer points in unions of faces whose
dimension does not exceed
\begin{equation*}
\textstyle
pqr - \binom{p}{2} - \binom {q}{2} - (p+q+r) + 1.
\end{equation*}
\end{teor}

\begin{obses}
(1) The significance of Theorem~\ref{teor:kron-eficiente} is that the sets on the right-hand side
of~\eqref{ecua:cancelacion geometrica} have dimension one less than the dimension of
the polytopes on the left-hand side.
Also, each face in $F^+_\ell(\la, \mu, \tau)$ or in $F^-_\ell(\la, \mu, \ol\tau)$ requires
fewer equations for its definition than the number required in the definition of
a column-row polytope.
However, in each union of faces one has to take into account the intersections of
these faces.

(2) It follows from (1) that the complexity of the computation of the right-hand side
of~\eqref{ecua:crdif-pq} is lesser than the complexity of the computation of the right-hand
side of~\eqref{ecua:kron-alternante}.

(3) As we noted in the introduction, when $r=2$ the formula in Theorem~\ref{teor:kron-eficiente}
permits one to compute Kronecker coefficients by counting integer points in polytopes whose dimension
does not exceed~\eqref{ecua:dimension balver}.
We expect to obtain an analogous formula for the case $r=3$.

(4) There is something mysterious in the column-row polytopes in that one can compute
the same Kronecker coefficient using different faces of the same polytopes depending on
the value of $\ell$.
\end{obses}

\section{Proofs of lemmas in Section~\ref{sec:insertion of words}} \label{sec:proofs}

\begin{proof}[Proof of Lemma~\ref{lema:cano-palabra}]
One implication is straightforward: just perform column insertion of $w$.
For the other, assume $P(w)= \cano\la$.
Let $ v = w_{\rm col}(\cano\la)$ the column word of $\cano\la$.
Then $v$ is a reverse lattice word and $P(v) = \cano\la$.
Thus $v$ and $w$ are Knuth equivalent.
Then Lemma~2 in page 66 from~\cite{ful} implies that $w$ is a reverse lattice word.
\end{proof}

\begin{proof}[Proof of Main Lemma~\ref{lema:cano-cano}]
Let
\begin{equation*}
\left(
\begin{matrix}
u_1 & \cdots & u_r \\
v_1 & \cdots & v_r
\end{matrix}
\right)
\end{equation*}
be the lexicographic two-row array associated to $B$.
Then (\cite[p.~198]{ful})
\begin{equation*}
P = P(\palabra v1r) = v_1 \rightarrow \cdots \rightarrow v_r \rightarrow \vacio.
\end{equation*}
By Lemma~\ref{lema:cano-palabra}, $P = \cano \mu$ if and only if $\palabra vr1$
is a lattice word.
This is equivalent to conditions (1.i) and (1.ii).
Thus (1) follows.
Let
\begin{equation*}
\left(
\begin{matrix}
x_1 & \cdots & x_r \\
y_1 & \cdots & y_r
\end{matrix}
\right)
\end{equation*}
be the lexicographic two-row array associated to the transpose matrix $\transpuesta B$.
Then, by the Symmetry Theorem for the RSK correspondence (\cite[p.~40]{ful}),
\begin{equation*}
Q = P(\palabra y1r) = y_1 \rightarrow \cdots \rightarrow  y_r \rightarrow \vacio.
\end{equation*}
By Lemma~\ref{lema:cano-palabra}, $Q = \cano \la$ if and only if $\palabra yr1$
is a lattice word.
This is equivalent to conditions (2.i) and (2.ii).
Thus (2) follows.
\end{proof}

\begin{proof}[Proof of Corollary~\ref{coro:cano-cano}]
The implication $(\implicarev)$ is straightforward from Lemma~\ref{lema:cano-cano}.
We prove $(\implica)$.
Due to the symmetry of the RSK correspondence  we may assume that $p \le q$.
Part (i) follows directly from Lemma~\ref{lema:cano-cano}.
For (ii) we observe that from Lemma~\ref{lema:cano-cano}, part (1.ii), one has
\begin{equation*}
b_{p,1} \ge b_{2,p-1} \ge \cdots \ge b_{1,p}.
\end{equation*}
Similarly, from Lemma~\ref{lema:cano-cano}, part (2,ii), and part (i), already proved, one has
\begin{equation*}
b_{1,p} \ge b_{2,p-1} \ge \cdots \ge b_{p,1}.
\end{equation*}
Therefore, $b_{p,1} = b_{2,p-1} = \cdots = b_{1,p}$.
Next, we consider the inequalities from parts (1) and (2) of Lemma~\ref{lema:cano-cano}
that involve summations with two terms on each side.
Since the entries $b_{i,j}$ in the diagonal $i+j=p+1$ are all equal, we conclude that
the entries $b_{i,j}$ in the diagonal $i+j=p$ are all equal.
We proceed in this way until we prove that $b_{1,2} = b_{2,1}$.
\end{proof}

\vskip 2pc
\textbf{Acknowledgments}

\bigskip
The first author would like to Aram Bingham for many fruitful conversations.


\vskip 2.5pc
{\small
\begin{thebibliography}{WW00}

\bibitem{av12} D. Avella-Alaminos and E. Vallejo,
Kronecker products and the RSK correspondence,
\emph{Discrete Math.}\ \textbf{312} (2012), 1476--1486.

\bibitem{bavew18}  V. Baldoni and M. Vergne,
Computation of dilated Kronecker coefficients (with an appendix by M. Walter),
\emph{J. Symbolic Comput.}\ \textbf{84} (2018), 113--146.

\bibitem{ba08} A.~Barvinok.
``Integer points in polyhedra'',
Zurich Lectures in Advanced Mathematics.
European Mathematical Society, 2008.

\bibitem{bp99} A.~Barvinok and J.~E.~Pommersheim,
An algorithmic theory of lattice points in polyhedra,
in ``New Perspectives in Algebraic Combinatorics'' (Berkeley, CA, 1996--97),
Math. Sci. Res. Inst. Publ. \textbf{38} (1999), 91--147.
Cambridge Univ. Press, Cambridge.

\bibitem{biva} A.~Bingham and E.~Vallejo,
On the computation of Kronecker coefficients II: when one partition has two parts,
in preparation.

\bibitem{cdw12} M. Christandl, B. Doran and M. Walter,
Computing multiplicities of Lie group representations,
in ``Proceedings of 2012 IEEE 53rd Annual Symposium of Foundations of Computer Science. FOCS 12'',
(2012), 639--648.

\bibitem{cm93} M. Clausen and H. Meier,
Extreme irreduzible Konstituenten in Tensordarstellungen symmetrischer Gruppen,
\emph{Bayreuther Math. Schriften}\ \textbf{45} (1993), 1--17.

\bibitem{delokim} J. A. De Loera and E. D. Kim,
Combinatorics and geometry of transportation polytopes: An update,
in ``Discrete geometry and algebraic combinatorics'',
Contemporary Mathematics \textbf{625}, AMS, (2014), 37--76.

\bibitem{don89} I. F. Donin,
Decompositions of tensor products of representations of the symmetric group
and symmetric and exterior powers of the adjoint representation of $\mathfrak g \mathfrak l (N)$,
\emph{Soviet Math. Dokl.}\ \textbf{38} (1989), 654--658.

\bibitem{dvi93} Y. Dvir,
On the Kronecker product of $S_n$ characters,
\emph{J. Algebra}\ \textbf{154} (1993), 125--140.

\bibitem{er90}
\"O. E\~gecio\~glu and J.B. Remmel,
A combinatorial interpretation of the inverse Kostka matrix,
\emph{Linear and Multilinear Algebra}\ \textbf{26} (1990), 59--84.

\bibitem{ful} W. Fulton,
``Young Tableaux,"
London Math. Soc. Student Texts 35,
Cambridge Univ. Press 1997.

\bibitem{ikenpak} Ch. Ikenmeyer and I. Pak,
What is in \# P and what is not?,
In 2022 IEEE 63rd Annual Symposium on Foundations of Computer Science (FOCS),
pp. 860--871. IEEE, 2022.

\bibitem{kn70} D.~E. Knuth,
Permutations, matrices, and generalized Young tableaux,
\emph{Pacific J. Math.}\  \textbf{34} (1970), 709--727.

\bibitem{liri34} D.~E. Littlewood and A.~R. Richardson,
Group characters and algebra,
\emph{Phil. Trans. Royal Soc. A (London)}\ \textbf{233} (1934), 99--141.

\bibitem{m95} I.~G. Macdonald,
``Symmetric functions and Hall polynomials'',
2nd. edition, Oxford Univ. Press, Oxford, 1995.

\bibitem{mrs} M. Mishna, M. Rosas and S. Sundaram,
Vector partition functions and Kronecker coefficients,
\emph{J. Phys A: Math. Theor.}\ \textbf{54} (2021) 205204.

\bibitem{mitr} M. Mishna and S. Trandafir,
Estimating and computing Kronecker coefficients: a vector partition function approach,
{\tt arXiv:2210.12128v1 [math.CO] 21 Oct 2022}.

\bibitem{mu38} F. D. Murnaghan,
The analysis of the Kronecker product of irreducible representations of the symmetric group,
\emph{Amer. J. Math.}\ \textbf{60} (1938), 761--784.

\bibitem{pak-intp-2022} I. Pak,
What is a combinatorial interpretation?,
{\sf arXiv:2209.06142v1 [math.CO]} 13 Sep 2022,
to appear in \emph{Proc. Open Problems in Algebraic Combinatorics}.

\bibitem{papa17} I. Pak and G. Panova,
On the complexity of computing Kronecker coefficients,
\emph{comput. complex.}\ \textbf{26} (2017), 1--36.

\bibitem{papa20} I. Pak and G. Panova,
Bounds on Kronecker coefficients via contingency tables,
\emph{Lin. Algebra Appl.}\ \textbf{602} (2020), 157--178.

\bibitem{pv05} I. Pak and E. Vallejo,
Combinatorics and geometry of Littlewood-Richardson cones,
\emph{European Journal of Combinatorics}\ \textbf{26} (2005), 995--1008.

\bibitem{reg80} A. Regev,
The Kronecker product of $S_n$ characters and an $A \otimes B$ theorem for
Capelli identities,
\emph{J. Algebra}\ \textbf{66} (1980), 505--510.

\bibitem{rt54} G. de B. Robinson and O.~E. Taulbee,
The reduction of the inner product of two irreducible representations of $\sime n$,
\emph{Proc. Nat. Acad. Sci. U.S.A.}\ \textbf{40} (1954), 723--726.

\bibitem{sa01} B.~E. Sagan,
``The Symmetric Group," 2nd. ed.,
Graduate Texts in Mathematics 203, Springer Verlag, 2001.

\bibitem{st99} R.~P. Stanley,
``Enumerative Combinatorics, Vol. 2'',
Cambridge Studies in Advanced Mathematics 62. Cambridge Univ. Press, 1999.

\bibitem{stan00} R.~P. Stanley,
Positivity problems and conjectures in algebraic combinatorics, in
``Mathematics: frontiers and perspectives'', AMS, Providence, RI, 2000, 295--319.

\bibitem{va99} E. Vallejo,
Stability of Kronecker products of irreducible characters of the
symmetric group,
{\em Electron. J. Combin}\ \textbf{6} (1999) Reseach Paper 39, 7 pp.
(electronic).

\bibitem{va00} E. Vallejo,
Plane partitions and characters of the symmetric group,
\emph{J. Algebraic Comb.}\ \textbf{11} (2000), 79--88.

\bibitem{va03} E. Vallejo,
Minimal matrices and minimal components in Kronecker products,
in Formal Power Series and Algebraic Combinatorics'03.
Poster session.
Link\"oping University, Sweden, 2003.

\bibitem{va09} E. Vallejo,
A stability property for coefficients in
Kronecker products of complex $S_n$ characters,
{\em Electron. J. Combin}\ \textbf{16} (2009) \#N22, 8 pp.

\bibitem{va14} E. Vallejo,
A diagrammatic approach to Kronecker squares,
\emph{J. Combin. Theory, Ser. A}\ \textbf{127} (2014), 243--285.

\bibitem{vakron2} E. Vallejo,
On the computation of Kronecker coefficients III: The facets of the column-row cone.
In preparation.

\end{thebibliography}
}
\end{document}